\documentclass[11pt,reqno]{amsart}

\usepackage[english]{babel}
\usepackage[utf8]{inputenc}
\usepackage[T1]{fontenc}
\usepackage{lmodern} \normalfont %to load T1lmr.fd 
\DeclareFontShape{T1}{lmr}{bx}{sc} { <-> ssub * cmr/bx/sc }{}
\usepackage{microtype}	% for improved spacing between words and letters

\usepackage{amssymb}
\usepackage{amsmath}
\usepackage{amsthm}
\usepackage{mathtools}
\usepackage{mathrsfs}
\usepackage{accents} % correct spacing for accents in sub- and superscript
\usepackage{etoolbox}
\usepackage{siunitx}
\usepackage{dsfont} % for special 1, \amthds{1}
\usepackage{graphicx}
\usepackage{color}
\usepackage[dvipsnames]{xcolor}
\usepackage{tikz}
\usetikzlibrary{calc,positioning,shapes}
\usetikzlibrary{patterns,decorations.pathmorphing,decorations.markings}
\usepackage{pgfplots}
\pgfplotsset{compat=newest}
\usepackage[margin=10pt,font=small,labelfont=bf,labelsep=endash]{caption}
\usepackage{subcaption}

\usepackage{booktabs}
\usepackage{paralist}
\usepackage{soul}

\numberwithin{equation}{section}

\usepackage{algorithm}
\usepackage{algorithmicx}
\usepackage{algpseudocode}

\usepackage{cite}
\usepackage{multirow} % for multirows in tabular environment
\usepackage{enumitem} % for labelling items in enumerate
\setlist[enumerate]{label=(\roman*)}

\usepackage[colorlinks,citecolor=black,urlcolor=black]{hyperref}
	\hypersetup{allcolors=black}
\usepackage[nameinlink,noabbrev]{cleveref}

\theoremstyle{plain}
\newtheorem{theorem}{Theorem}[section]
\newtheorem{proposition}[theorem]{Proposition}
\newtheorem{lemma}[theorem]{Lemma}
\newtheorem{corollary}[theorem]{Corollary}
\newtheorem{remark}[theorem]{Remark}

\DeclareMathOperator*{\argmin}{arg\,min} %argmax
\definecolor{mycolor1}{rgb}{0.00000,0.44700,0.74100}% blue
\definecolor{mycolor2}{rgb}{0.85000,0.32500,0.09800}% red
\definecolor{mycolor3}{rgb}{0.92900,0.69400,0.12500}% orange/yellow
\definecolor{mycolor4}{rgb}{0.46600,0.67400,0.18800}% green
\definecolor{mycolor5}{rgb}{0.49400,0.18400,0.55600}% purple

\title{Local Well-Posedness of the Mortensen Observer}
\author{Tobias Breiten${}^\dagger$ \and Jesper Schr\"oder${}^\dagger$}
\address{${}^{\dagger}$  Institute of Mathematics MA\,{}4-4, Technical University Berlin, Stra\ss e des 17.~Juni 136, 10623 Berlin, Germany}
\email{tobias.breiten@tu-berlin.de}
\email{j.schroeder@tu-berlin.de}
\date{\today}

\begin{document}

\begin{abstract}
	The analytical background of nonlinear observers based on minimal energy estimation is discussed. It is shown that locally the derivation of the observer equation based on a trajectory with pointwise minimal energy can be done rigorously. The result is obtained by a local sensitivity analysis of the value function based on Pontryagin's maximum principle and the Hamilton-Jacobi-Bellman equation. The consideration of a differential Riccati equation reveals that locally the second derivative of the value function is a positive definite matrix. The local convexity ensures existence of a trajectory minimizing the energy, which is then shown to satisfy the observer equation.
\end{abstract}

\maketitle

 
\section{Introduction}
We consider a nonlinear perturbed dynamical system of the form
\begin{equation}\label{eq:state_intro}
	\begin{aligned}
		\dot{x}(t) &= Ax(t) + G(x(t)\otimes x(t)) +Fv(t) , \quad t \in (0,T], \\
		x(0)&= x_0 + \eta
	\end{aligned}
\end{equation}
where $A \in \mathbb R^{n,n}, G\in \mathbb R^{n,n^2},F\in \mathbb R^{n,m}, x_0\in \mathbb R^n$ are known and $v\in L^2(0,t;\mathbb R^m),\eta \in \mathbb R^n$ are assumed to be (deterministic) disturbances. Let us emphasize that the specific quadratic structure of the nonlinearity is mainly due to technical simplification rather than a necessary requirement for the results obtained in this article. In particular, by a process sometimes called \emph{lifting}, a considerably more general class of nonlinear systems can be embedded in a structure of the form \eqref{eq:state_intro}, see, e.g., \cite{BeBre15,Gu11}. For the system \eqref{eq:state_intro}, let us consider a (disturbed) linear observation of the form 
\begin{align}\label{eq:output_intro}
	y(t)=Cx(t) + \mu(t),
\end{align}
with $C\in \mathbb R^{r,n}$ and $\mu \in L^2(0,T;\mathbb R^r)$. The goal of this article is the theoretical analysis of the well-known (see \cite{Mor68}) \emph{Mortensen observer} 
\begin{equation}\label{eq:mortensen_intro}
	\begin{aligned}
		\dot{\widehat{x}}(t)&= A\widehat{x}(t) + G(\widehat{x}(t)\otimes \widehat{x}(t)) + \alpha \nabla^2_{\xi\xi}\mathcal{V}(t,\widehat{x}(t),y)^{-1} C^\top (y(t)-C\widehat{x}(t))  ,\quad t \in (0,T], \\
		\widehat{x}(0)&=x_0
	\end{aligned}
\end{equation}
which can (formally) be derived by pointwise minimization of the minimal value function of an optimal control problem associated with \eqref{eq:state_intro}. More precisely, for given $t\in (0,T]$ and $\xi \in \mathbb R^{n}$ we have that
\begin{equation}\label{eq:vf_intro}
	\begin{aligned}
		\mathcal{V}(t,\xi,y) = \inf_{\substack{x \in H^1(0,t;\mathbb{R}^n)\\ v \in L^2(0,t;\mathbb{R}^m)}} J(x,v;t,y) &\coloneqq \frac{1}{2} \left\Vert x(0) - x_0 \right\Vert^2 + \frac{1}{2} \int_0^t \Vert v \Vert^2 + \alpha \Vert y-Cx \Vert^2 \,\mathrm{d}s, \\ \text{s.t. \ }     e(x,v;t,\xi) & \coloneqq (\dot{x} - A x - G (x \otimes x) - F v, x(t)  - \xi) = 0.
	\end{aligned}
\end{equation}
An integral part of our analysis is the discussion of the corresponding time-dependent non-homogeneous Hamilton-Jacobi-Bellman (HJB) equation
\begin{equation}\label{eq:hjb_intro}
	\begin{aligned}
		\partial_t \mathcal{V}(t,\xi,y) 
		&= - \left\langle \nabla_\xi \mathcal{V}(t,\xi, y) , A\xi + G(\xi \otimes \xi) \right\rangle - \frac{1}{2} \left\Vert F^\top \nabla_\xi \mathcal{V}(t,\xi,y) \right\Vert^2 + \frac{\alpha}{2} \left\Vert y(t) - C \xi \right\Vert^2, \\
		\mathcal{V}(0,\xi,y)&= \frac{1}{2}\|\xi-x_0\|^2.
	\end{aligned}
\end{equation}

The observer \eqref{eq:mortensen_intro} was initially proposed as a \emph{maximum likelihood filter} in \cite{Mor68} where the author considered a more general setup as we do here, also allowing for nonlinear output operators. As was already noted by Mortensen, it generalizes the Kalman-Bucy filter \cite{Kal60,KalB61} which is obtained in \eqref{eq:mortensen_intro} for the special case of linear dynamics, i.e., when $G=0$ in \eqref{eq:state_intro}. In contrast to similar works on nonlinear (stochastic) filtering \cite{ItoX00,Kus62,Kus67,Zak69}, the maximum likelihood filter relies on a deterministic interpretation of the unknown disturbances $\eta,v$ and $\mu$ thereby connecting it to a deterministic optimal control problem of the form \eqref{eq:vf_intro}, see also \cite{Wil04} for a thorough discussion of the linear case. Moreover, the structure \eqref{eq:mortensen_intro} resembles a state-dependent Luenberger observer \cite{Lue71}. In the literature, \eqref{eq:mortensen_intro} is also referred to as a \emph{minimum energy estimator} \cite{Hij80,Kre79} or simply the \emph{Mortensen observer/estimator} \cite{Fle97,Moi18}. While initially derived on a formal level, in \cite{Fle97} the author provides a rigorous analysis of the optimal control problem as well as the value function under the assumption that the arising nonlinearities are continuously differentiable with globally bounded derivatives, cf.~also the more detailed general exposition in \cite{FleS06}. In \cite{Kre03a}, it is shown that if the nonlinearities are globally Lipschitz continuous, then the observer state $\widehat{x}(t)$ in \eqref{eq:mortensen_intro} converges asymptotically to the true state $x(t)$ as $t\to \infty$. Due to the value function in \eqref{eq:vf_intro} suffering from the curse of dimensionality, different approximations of the Mortensen observer have been proposed and theoretically analyzed. Let us mention \cite{BarBJ88} which discusses convergence results for nonlinearities with a globally bounded second derivative which is used within an approximate minimum energy estimator. Related to that, in \cite{Kre03} the author discusses convergence results for the extended Kalman filter which is known to be a (first order) polynomial approximation to the Mortensen observer, see also \cite{Mor68}. More recently, in \cite{Moi18} a discrete-time version of the Mortensen observer was analyzed under the assumption of affine dynamics. Let us also point to a numerical realization of the Mortensen observer by a neural network based approximation approach that was studied in \cite{BreKu21}.

Looking at the observer equation \eqref{eq:mortensen_intro} it becomes apparent that for the discussion of well-posedness of the observer trajectory it is essential to ensure existence of the Hessian $\nabla_{\xi\xi}^2 \mathcal{V}$ and further show that it is an invertible matrix. In general however, value functions are notoriously non-smooth, therefore showing the required smoothness is one of the main challenges of this work. For example, the seminal work \cite{CraL83} displays, that HJB equations as in \eqref{eq:hjb_intro} generally do not allow for a classical notion of a solution. Early works on regularity of value functions for finite and infinite-dimensional systems can be found in, e.g.,  \cite{CanF91,CanF92}. In the latter references, the considered nonlinearities are assumed to be Lipschitz continuous and exhibit a linear growth condition w.r.t.~the state variable. In view of the term $G(x\otimes x)$, these results do not directly apply here. Let us also refer to \cite{CanFra13,CanF14} where the authors show local regularity results for HJB solutions based on pointwise regularity. In the context of (unconstrained) infinite-horizon  control problems, in \cite{BreKP19,BreKuPf19} it has been shown that the associated (time-independent) value function is infinitely often differentiable. The problem of finite-horizon control problems under state and control constraints was discussed in \cite{Mal83,MalM01,MalM98}. Here the authors were able to show that the solutions of a parameterized problem depend on said parameter in a differentiable fashion. However, these results do not yield time-differentiability of the value function considered in this work, because the controls are assumed to be essentially bounded. Further \cite{MalM01,MalM98} explicitly assume the parameter to be independent of time. 

In this article, we will employ some of the ideas from \cite{BreKP19,BreKuPf19} and perform a sensitivity analysis for appropriately chosen time-dependent finite-horizon non-homogeneous control problems as in \eqref{eq:vf_intro}. Our main results can be summarized as follows:
\begin{itemize}   
	\item[(i)] Based on a nominal trajectory obtained for the undisturbed dynamics, we define a       local neighborhood on which the HJB equation \eqref{eq:hjb_intro} has a classical            solution, see \Cref{Theorem: HJB holds}. 
	\item[(ii)] For the Hessian $\nabla^2_{\xi\xi}\mathcal{V}$ of the value function $\mathcal{V}$, in \Cref{Proposition: DRE in zero without IV} we analyze a specific differential Riccati equation whose solution we show to be positive definite. As a consequence, in \Cref{Theorem: Hessian invertible close to zero}, we locally extend this result implying the (local) positive definiteness of $\nabla^2_{\xi\xi}\mathcal{V}(t,\xi,y)$.  
	\item[(iii)] In Lemma \ref{lem: VFminimizer} we establish an upper bound for minimizers of $\mathcal{V}(t,\cdot,\omega)$. Together with the local strict convexity of $\mathcal{V}$ a unique minimizer can be identified as a solution of $\nabla_\xi \mathcal{V}(t,\cdot,\omega) = 0$, cf. Proposition \ref{prop: WDargmin}.
	\item[(iv)] For sufficiently small and continuous $\omega$ the observer equation can then be derived in a rigorous manner. Via a density argument the result is transferred to less regular data. The main result is stated in \Cref{thm: MainThm}. It shows that if initial perturbation and dynamics as well as output disturbances are sufficiently small, the Mortensen observer \eqref{eq:mortensen_intro} is well-defined. 
\end{itemize}  

The technique of showing regularity of the value function by an application of the inverse mapping theorem that we deploy is a well known tool in the context of sensitivity analysis. A thorough display stating general results can be found for example in \cite{ItoK08}. The results presented there could be used to show first order differentiability of the value function with respect to space and time. However, this work further requires time continuity of the spatial derivatives of the value function up to order three. To the best of the authors' knowledge this smoothness can not be obtained directly with any of the available sources, therefore this work contains the technical sensitivity analysis of the specific problem using only elementary results.
The above mentioned time continuity of $\nabla_\xi \mathcal{V}$, $\nabla^2_{\xi\xi} \mathcal{V}$ and $\nabla^3_{\xi^3} \mathcal{V}$ is shown using an argument of time uniform convergence of the difference quotients utilizing a bound for the spatial derivatives, see Proposition \ref{Proposition: Space derivatives continuous in time}. Hence it is crucial that all constants in this work remain independent of time. While this poses a challenge throughout the article, it also allows for a stronger result, see Remark \ref{rem: MainThm}.\\

\textbf{Notation.}
If not mentioned otherwise, $\Vert \cdot \Vert$ will denote the Euclidean norm on $\mathbb{R}^d$, where the dimension $d$ varies. The associated scalar product is denoted by $\langle \cdot, \cdot \rangle$. The $i$-th canonical basis vector is denoted by $e_i$. Further we denote by $I_d$ the identity matrix of dimension $d$ and the matrix spectral norm is denoted by $\Vert \cdot \Vert_2$. The Kronecker product of two matrices $A$ and $B$ or two vectors $x$ and $y$ is denoted by $A \otimes B$ and $x \otimes y$, respectively. For $1 \leq p \leq \infty $ we denote by $L^p(0,T;\mathbb{R}^d)$ the Lebesgue spaces, while $H^1(0,T;\mathbb{R}^d)$ denotes the Sobolev space of functions with a first weak derivative in $L^2(0,T;\mathbb{R}^d)$. Furthermore $C([0,T];\mathbb{R}^d)$ and $C^k([0,T];\mathbb{R}^d)$ denote the spaces of functions $f \colon [0,T] \to \mathbb{R}^d$ that are continuous and $k$ times continuously differentiable, respectively. For $d = 1$ the image space is dropped in the notation. The space of all functions from $C^\infty([0,T];\mathbb{R}^n)$ with compact support is denoted by $C_0^\infty([0,T];\mathbb{R}^n)$. All mentioned spaces are equipped with their standard norms. The space of solution trajectories will be denoted by $V_t$ and is given by $H^1(0,t;\mathbb{R}^n)$ equipped with the norm $\Vert \cdot \Vert_{V_t} = \max \left(\Vert \cdot \Vert_{H^1(0,t;\mathbb{R}^n)} , \Vert \cdot \Vert_{L^\infty(0,t;\mathbb{R}^n)}\right)$.

For two Banach spaces $X$ and $Y$ their Cartesian product is denoted by $X \times Y$ and equipped with the norm $\Vert (x,y) \Vert_{X \times Y} = \max \left( \Vert x \Vert_X, \Vert y \Vert_Y \right)$. Further $L(X,Y)$ denotes the space of linear and bounded mappings from $X$ to $Y$. For an element $x \in X$ and a real number $\epsilon > 0$ the open ball of radius $\epsilon$ around $x$ is denoted by $\mathcal{U}_\epsilon(x)$. The weak convergence of a sequence $x_k$ to some $x$ is denoted as $x_k \rightharpoonup x$, for $k \to \infty$. For a function $f \colon X \to Y$ its Fr\'{e}chet derivative is denoted by $Df$. For a function $f \colon X_1 \times X_2 \to Y$, where $X_1$ and $X_2$ are Banach spaces, the partial Fr\'{e}chet derivative with respect to the first variable is denoted by $D_{x_1}f$.  Higher order and mixed partial derivatives are denoted with appropriate indices, e.g., $D^2_{x_1 x_2}f$.

Throughout this work we use $c$ as a generic constant. 

%%%%%%%%%%%%%%%%%%%%%%%%%%%%%%%%%%%%%%%%%%%%%%%%%%%%%
\section{Preliminaries}\label{Section: Preliminaries}
%%%%%%%%%%%%%%%%%%%%%%%%%%%%%%%%%%%%%%%%%%%%%%%%%%%%%
For this entire work let $T > 0$ be fixed. In order to formulate the state equation of interest assume that for some given initial value $x_0 \in \mathbb{R}^n$ and matrices $A \in \mathbb{R}^{n,n}$ and $G \in \mathbb{R}^{n,n^2} $ the problem
\begin{equation*}
	\begin{aligned}
		\dot{x}(t) &= Ax(t) + G(x(t) \otimes x(t)),\\
		x(0) &= x_0,
	\end{aligned}
\end{equation*}
admits a unique solution $\Tilde{x} \in V_T$. This so-called nominal trajectory $\Tilde{x}$ will play an important role throughout this work.
The state equation is formulated on an interval $[0,t]$, where $t \in (0,T]$ is fixed. It reads 
\begin{equation}\label{StateEquationExplicit}
	\begin{aligned}
		\dot{x}(s) &= A x(s) + G (x(s) \otimes x(s)) + Fv(s), \\
		x(t) &= \Tilde{x}(t) + \xi, 
	\end{aligned}
\end{equation}
where the matrix $F \in \mathbb{R}^{n,m}$ is given. Furthermore the disturbance $v \in L^2(0,t; \mathbb{R}^m)$ is for now fixed. One searches for a solution $x \in V_t$ that satisfies the first equation almost everywhere in $[0,t]$. Note that this formulation is not consistent with \eqref{eq:vf_intro}. For technical reasons in the following $\xi$ will not describe the potential state of the observer but the difference of observer trajectory and nominal trajectory. This can be considered as a change of coordinates. Instead of considering neighborhoods around $\Tilde{x}$, we deal with neighborhoods of zero.

We further want to emphasize that the dynamics considered in this work are given by a disturbed initial value problem, see \eqref{eq:state_intro}. This specifically means that the dynamics evolve forward in time, as does the observer trajectory characterized by \eqref{eq:mortensen_intro}. However, the technical discussions of this work are concerned with the state equation of the optimal control problem which is given as a final value problem and therefore is associated with dynamics evolving backwards in time, see \eqref{eq:vf_intro} and \eqref{StateEquationExplicit}. This discrepancy stems from the idea to define the observer via an energy minimization and sets our discussion apart from the dominant literature. There the value function and HJB equation is usually discussed in the context of a forward problem. 

Before the state equation can be discussed, a result about a general linear equation is presented. It is of the form
\begin{equation}\label{LinearEquationInitialValueGeneral}
	\begin{aligned}
		\dot{x}(s) &= B_t(s) x(s) + f(s),\\
		x(0) &= \xi,
	\end{aligned}
\end{equation}
where $f \in L^2(0,t;\mathbb{R}^n)$ is fixed. The system matrix $B_t$ is not only time-dependent, but also depends on the interval $[0,t]$ on which the equation is considered. However it will be assumed that the spectral norm of the system matrix is bounded uniformly in time and in the parameter $t$. Since the right hand side is measurable and integrable in $s$, and fulfills a generalized Lipschitz condition in $x$, this problem admits a unique solution $x \in V_t$. The following lemma presents an estimate for the solution.
\begin{lemma}\label{Lemma: Estimate linear equation}
	Let $s \mapsto B_t(s) \in \mathbb{R}^{n,n}$ be a matrix-valued function depending on the parameter $t \in (0,T]$. Let $t \in (0,T]$ and assume that there exists a constant $\kappa > 0$ independent of $t$ and $s$ such that for all $s \in [0,t]$ it holds $\Vert B_t(s) \Vert_2 \leq \kappa$. Then there exists a constant $\bar{c}(\kappa) > 0$ independent of $t$, $\xi$ and $f$ such that the unique solution $x \in V_t$ of \eqref{LinearEquationInitialValueGeneral} satisfies
	\begin{equation*}
		\Vert x \Vert_{V_t} \leq \bar{c}(\kappa) \left( \Vert \xi \Vert + \Vert f \Vert_{L^2(0,t;\mathbb{R}^n)} \right).
	\end{equation*}
\end{lemma}
\begin{proof}
	The proof is done in a standard way by testing the equation with its solution and integrating over time. Utilizing Young's inequality, Gronwall's inequality and the representation of $\dot{x}$ given by the equation yields the assertion.
\end{proof}
Via a time transformation it can be shown that the same result holds, if one considers a finite value problem instead of the initial value problem.\\
\section{The optimal control problem}\label{Section: The optimal control problem}
The main result of this work is the proof of existence of a solution to the observer equation
\begin{equation}\label{eq:ObsShif}
	\begin{aligned}
		\dot{\widehat{x}}(t)&= A'(t)\widehat{x}(t) + G(\widehat{x}(t)\otimes \widehat{x}(t)) + \alpha \nabla^2_{\xi\xi}\mathcal{V}(t,\widehat{x}(t),\omega)^{-1} C^\top (\omega(t)-C\widehat{x}(t))  ,\quad t \in (0,T], \\
		\widehat{x}(0)&=0,
	\end{aligned}
\end{equation}
where $A'(t) = A + G(\Tilde{x}(t) \otimes I_n) + G(I_n \otimes \Tilde{x}(t))$. The function $\mathcal{V}(t,\xi,\omega)$ is the value function associated with the control problem 
\begin{equation}\label{Control Problem}
	\begin{aligned}
		\inf_{(x,v) \in V_t \times L^2(0,t;\mathbb{R}^m)} J(x,v;t,\omega) &\coloneqq \frac{1}{2} \left\Vert x(0) - x_0 \right\Vert^2 + \frac{1}{2} \int_0^t \Vert v \Vert^2 + \alpha \Vert \omega - C \left(x-\Tilde{x}\right) \Vert^2 \,\mathrm{d}s, \\ \text{s.t. }
		e(x,v;t,\xi) &\coloneqq (\dot{x} - A x - G (x \otimes x) - F v, x(t) - \Tilde{x}(t) - \xi) = 0,
	\end{aligned}
\end{equation}
where $t\in (0,T]$, $\alpha > 0$, $C \in \mathbb{R}^{r,n}$, $\omega \in L^2(0,t;\mathbb{R}^r)$ and $\xi \in \mathbb{R}^n$ are fixed. The given $x_0 \in \mathbb{R}^n$ is the initial condition of the nominal trajectory $\tilde{x}$.

Similar to the coordinate change performed for $\xi$, we utilize a transformation for the measured data. The control problem and the value function are not formulated in terms of the measured data $y$. Instead they depend on the difference of measured data and modeled output. More precisely, we introduce the variable $\omega = y - C\Tilde{x}$. Now the undisturbed model considered on $[0,t]$ corresponds to $\xi = 0$ and $\omega = 0$ instead of $\xi = \tilde{x}(t) $ and $y = C\Tilde{x}$. In this sense \eqref{eq:ObsShif} can be understood as an adjusted version of \eqref{eq:mortensen_intro}. Specifically, the solution trajectories are shifted by $\Tilde{x}$. Since both formulations are equivalent, it suffices to show results for the shifted one.\\ 

The purpose of Sections 3 to 5 is to prove sufficient regularity of $\mathcal{V}$ such that $\nabla^2_{\xi\xi}\mathcal{V}$ exists and is an invertible matrix. This however can only be achieved in a local sense, which results in the necessity for an upper bound on $\Vert \omega \Vert_{L^2(0,T;\mathbb{R}^r)}$ to ensure solvability of \eqref{eq:ObsShif}. More precisely, it will be shown that there exists a constant $\delta > 0$ such that the assertion holds for all $\omega$ satisfying $\Vert \omega \Vert_{L^2(0,T;\mathbb{R}^r)} < \delta$. In order to achieve this, it is essential that all constants describing locality are independent of time. This goal will pose a challenge throughout this work. The benefits of this effort are summarized in Remark \ref{rem: MainThm}.

In order to obtain regularity results concerning $\mathcal{V}$, one first has to perform a standard analysis of the optimal control problem. The remainder of this section is concerned with solvability of the state equation, existence of minimizers and the first order optimality condition expressed via the solution of the adjoint state equation.
\subsection{State equation}
First it will be shown that for $\xi \in \mathbb{R}^n$ and $v \in L^2(0,t;\mathbb{R}^m)$ with sufficiently small norms the state equation \eqref{StateEquationExplicit} admits a unique solution. For this subsection let $t \in (0,T]$ be fixed. First consider
\begin{equation}\label{ErrorEquationGeneral}
	\dot{x} = - A x - G(I_n \otimes \Tilde{x}_t) x - G(\Tilde{x}_t \otimes I_n) x - G(x \otimes x) - f,~
	x(0) = \xi,
\end{equation}
on the interval $[0,t]$, where $\Tilde{x}_t(s) \coloneqq \Tilde{x}(t-s)$ is a time transformation of the nominal trajectory on $[0,t]$.
\begin{proposition}\label{Proposition:ErrorSolvability}
	Let $t \in (0,T]$ and let $\bar{c} = \bar{c}(\Vert A \Vert_2 + 2 \Vert G \Vert_2\Vert \Tilde{x} \Vert_{L^\infty(0,T;\mathbb{R}^n)})$ be the constant from Lemma \ref{Lemma: Estimate linear equation}. Then there exists a constant $\delta_1^\prime > 0$ independent of $t$ such that for all $\xi \in \mathbb{R}^n$ and $f \in L^2(0,t;\mathbb{R}^n)$ satisfying
	\begin{equation*}
		\Vert \xi \Vert + \Vert f \Vert_{L^2(0,t;\mathbb{R}^n)} \leq \delta_1^\prime 
	\end{equation*}
	there exists a unique solution $x \in V_t$ of \eqref{ErrorEquationGeneral}. It satisfies 
	\begin{equation*}
		\Vert x \Vert_{V_t}
		\leq 2 \bar{c} \left( \Vert \xi \Vert + \Vert f \Vert_{L^2(0,t;\mathbb{R}^n)} \right).
	\end{equation*}
\end{proposition}
\begin{proof}
	The proof of existence and the bound is carried out via a fixed point argument on the set
	\begin{equation*}
		\mathcal{M} \coloneqq \left\{ x \in V_t \,\colon
		\Vert x \Vert_{V_t}
		\leq 2 \bar{c} \left( \Vert \xi \Vert + \Vert f \Vert_{L^2(0,t;\mathbb{R}^n)} \right) \right\}.
	\end{equation*}
	Such a strategy to prove existence of a solution can be found for example in \cite[Lemma 5]{BreKuPf19}. Uniqueness of the solution in $V_t$ is proven in a standard manner.
\end{proof}
\begin{remark}
	For an appropriate estimation of the quadratic term the estimates for both the sup norm and the $H^1$-norm are required. The need for the estimate of the sup norm can not be circumvented by the continuous embedding 
	\begin{equation*}
		\Vert x - \Tilde{x} \Vert_{L^\infty(0,t;\mathbb{R}^n)} \leq \check{c}(t) \Vert x - \Tilde{x} \Vert_{H^1(0,t;\mathbb{R}^n)}.
	\end{equation*}
	This is due to the unboundedness of the optimal embedding constant $\check{c}(t) = \sinh(t)^{-\tfrac{1}{2}}$ (see \cite{Mar83}) for $t$ approaching zero. Without the estimate for the sup norm $\delta_1$ would depend on time $t$ with the property $\delta_1(t) \to 0$ for $t \to 0$. This issue will come up frequently in the following discussions. Therefore estimates are presented for the $V_t$-norm instead of the $H^1$-norm.
\end{remark}
The unique solvability of the state equation follows immediately.
\begin{corollary}\label{Corollary: SolvabilityStateEquation}
	There exist constants $\delta_1 > 0$ and $\tilde{c} > 0$ independent of $t$ such that for any pair $(\xi, v) \in \mathbb{R}^n \times L^2(0,t;\mathbb{R}^m) $ satisfying $ \Vert \xi \Vert + \Vert v \Vert_{L^2(0,t;\mathbb{R}^n)} \leq \delta_1 $ the state equation \eqref{StateEquationExplicit} admits a unique solution $x \in V_t$. It satisfies
	\begin{equation*}
		\Vert x - \Tilde{x} \Vert_{V_t}
		\leq \Tilde{c} \left( \Vert \xi \Vert + \Vert v \Vert_{L^2(0,t;\mathbb{R}^m)} \right).
	\end{equation*}
\end{corollary}
\begin{proof}
	By performing a time transformation of \eqref{ErrorEquationGeneral} one concludes the unique solvability of 
	\begin{equation*}\label{EquationErrorExplicit}
		\dot{e} = A e + G(I_n \otimes \Tilde{x}) e + G(\Tilde{x} \otimes I_n) e + G(e \otimes e) + Fv,~
		e(t) = \xi
	\end{equation*}
	for sufficiently small $\xi$ and $f = Fv$. Then $x = \Tilde{x} + e$ is the desired solution and satisfies the estimate of the assertion.
\end{proof}
\subsection{Existence of a minimizer}\label{SubS:MinEx} Next it will be shown that for sufficiently small data the control problem admits a solution. Again the proof can be done using well known techniques and we limit its discussion to a rough sketch. For the general result we refer to \cite[Theorem 1.45]{HinetalS09} and for a very similiar statement see \cite[Lemma 8]{BreKuPf19}.
\begin{proposition}\label{Proposition: Existence}
	Let $t \in (0,T]$. There exists a constant $\delta_2 \in (0,\delta_1] $ independent of $t$ such that for all $(\xi,\omega) \in \mathbb{R}^n \times L^2(0,t;\mathbb{R}^r)$ satisfying $\max\left( \Vert \xi \Vert , \Vert \omega \Vert_{L^2(0,t;\mathbb{R}^r)}\right) \leq \delta_2$ the optimal control problem \eqref{Control Problem} admits a solution $(\bar{x},\bar{v}) \in V_t \times L^2(0,t;\mathbb{R}^m)$. Furthermore, there exists a constant $M_1>0$ independent of $t$, $\xi$ and $\omega$ such that any minimizing pair $(z,u) \in V_t \times L^2(0,t;\mathbb{R}^m)$ satisfies
	\begin{equation*}
		\max \left(\Vert z - \Tilde{x} \Vert_{V_t}, \Vert u \Vert_{L^2(0,t;\mathbb{R}^m)} \right) 
		\leq M_1 \max\left( \Vert \xi \Vert , \Vert \omega \Vert_{L^2(0,t;\mathbb{R}^r)} \right).\\
	\end{equation*}
\end{proposition}
\begin{proof}
	For now set $\delta_2 = \delta_1$. With Corollary \ref{Corollary: SolvabilityStateEquation} it follows that the uncontrolled state equation associated with $\xi$ admits a unique solution. Hence one obtains an upper bound for the infimum in \eqref{Control Problem} of the form $j \leq c \delta_2^2$. For a minimizing sequence $(x_k,v_k)$ one can assume $J(x_k,v_k;t,\omega) \leq 2 c \delta_2^2$ which yields $\Vert v_k \Vert_{L^2(0,t;\mathbb{R}^m)} \leq \sqrt{2} c \delta_2$. A reduction of $\delta_2$ leads to $\Vert \xi \Vert + \Vert v_k \Vert_{L^2(0,t;\mathbb{R}^m)} \leq \delta_2 + \sqrt{2} c \delta_2 \leq \delta_1$ allowing the application of Corollary \ref{Corollary: SolvabilityStateEquation} to the state equation controlled by $v_k$. Going over to appropriate subsequences of the bounded sequences $(x_k)$ and $(v_k)$ implies the existence of weak limits $\bar{x}$ and $\bar{v}$. Passing to the weak limit it is shown that $e(\bar{x},\bar{v},\xi,t) = 0$. In order to pass to the weak limit in the quadratic term one has to go over to a subsequence that converges strongly in $L^2(0,t;\mathbb{R}^n)$. This is possible because $H^1$ is compactly embedded in $L^2$. Since the cost functional is weakly sequentially lower semi continuous, $(\bar{x},\bar{v})$ is a minimizer. The estimate can be derived with the strategies applied to show boundedness of the minimizing sequence.
\end{proof}
\subsection{Necessary optimality condition}
In this subsection a first order necessary optimality condition of the optimal control problem will be derived.
\begin{proposition}\label{Proposition: FirstOrderOptimalityCondition}
	Let $t \in (0,T]$ be fixed and assume $\max\left( \Vert \xi \Vert, \Vert \omega \Vert_{L^2(0,t;\mathbb{R}^r)} \right) \leq \delta_2$ and let $(\bar{x},\bar{v}) \in V_t \times L^2(0,t;\mathbb{R}^m)$ be a minimizer of the corresponding optimal control problem \eqref{Control Problem}. Then there exists a unique adjoint state $p \in V_t$ satisfying
	\begin{align}
		\dot{p} &= -A^\top p - (\bar{x}^\top \otimes I_n)G^\top p - (I_n \otimes \bar{x}^\top)G^\top p + \alpha C^\top (\omega - C (\bar{x}-\Tilde{x})),\label{Adjoint1}\\
		p(0) &= x_0 - \bar{x} (0),\label{Adjoint2}\\
		\bar{v} + F^\top p &= 0\label{OptimalityConditionFormula},
	\end{align}
	where \eqref{Adjoint1} is fulfilled almost everywhere in $[0,t]$, \eqref{Adjoint2} is fulfilled in $\mathbb{R}^n$ and \eqref{OptimalityConditionFormula} holds in $L^2(0,t;\mathbb{R}^m)$. Furthermore there exists a constant $M_2 > 0$ independent of $t$, $\bar{x}$, $\bar{v}$, and $\omega$ such that it holds
	\begin{equation*}
		\Vert p \Vert_{V_t}
		\leq M_2 \max\left( \Vert \xi \Vert , \Vert \omega \Vert_{L^2(0,t;\mathbb{R}^r)} \right).
	\end{equation*}
\end{proposition}
\begin{proof}
	This standard result can be shown for example by an application of \cite[Proposition 1, Section 4.14]{Zei95AMS109}. 
\end{proof}

\section{Sensitivity analysis}\label{Section: Sensitivity analysis}

For $(\xi,\omega)$ satisfying $\max \left( \Vert \xi \Vert, \Vert \omega \Vert_{L^2(0,T;\mathbb{R}^r)} \right) \leq \delta_2$ and $t\in (0,T]$ the value function corresponding to the control problem \eqref{Control Problem} is defined as
\begin{equation*}
	\mathcal{V}(t,\xi,\omega) = \min\limits_{\substack{x \in V_t \\ v \in L^2(0,t;\mathbb{R}^m)}} J(x,v;t,\omega), ~~\text{ subject to: } ~e(x,v;t,\xi) = 0
\end{equation*}
with initialization defined by $\mathcal{V}(0,\xi,\omega) = \tfrac{1}{2} \Vert \xi \Vert^2$. Note that due to Proposition \ref{Proposition: Existence} it is well-defined.

The goal of this section is to show sufficient regularity of the value function to ensure that its Hessian $\nabla^2_{\xi\xi} \mathcal{V}$ is well-defined and differentiable in $t$.
\subsection{Regularity in space and measurement}\label{Subsection: Regularity in space and measurement}
The regularity of the value function with respect to the last two variables will be shown using the chain rule.
To that end it is shown that for sufficiently small data there exists a unique minimizing pair which depends on the data in a smooth fashion. For a fixed $t \in (0,T]$ this is achieved by an application of the implicit function theorem to the function $P_t \colon Y_t \times X_t \rightarrow Y_t$, $(w,h) \mapsto \Phi_t(h) - w$, where 
\begin{equation*}
	\begin{aligned}
		&X_t \coloneqq V_t \times L^2(0,t;\mathbb{R}^m) \times V_t \times L^2(0,t;\mathbb{R}^r),\\ 
		&Y_t \coloneqq L^2(0,t;\mathbb{R}^r) \times L^2(0,t;\mathbb{R}^n) \times \mathbb{R}^n \times L^2(0,t;\mathbb{R}^n) \times \mathbb{R}^n \times L^2(0,t;\mathbb{R}^m),
	\end{aligned}
\end{equation*}
and $\Phi_t: X_t \rightarrow Y_t$ is defined by
\begin{equation}\label{Phi Definition}
	\Phi_t(x,v,p,\omega) =
	\begin{pmatrix}
		\omega\\
		\dot{x} - Ax - G(x \otimes x) - Fv\\
		x(t) - \Tilde{x}(t)\\
		\dot{p} + A^\top p + (x^\top \otimes I_n)G^\top p + (I_n \otimes x^\top)G^\top p - \alpha C^\top(\omega - C(x-\Tilde{x}))\\
		p(0) - x_0 + x(0)\\
		v + F^\top p
	\end{pmatrix}.
\end{equation}
Note that it holds $\Phi_t(x,v,p,\omega) = (\omega,0,\xi,0,0,0)$ if and only if the triple $(x,v,p)$ satisfies the first order necessary optimality condition of the optimal control problem \eqref{Control Problem} associated with $\xi$ and $\omega$.

The strategy of applying the implicit function to the mapping $P_t$ is equivalent to an application of the inverse mapping theorem to $\Phi_t$. The specific version of the implicit function theorem that we apply is presented in \cite{Hol70}. It gives sufficient conditions under which the obtained neighborhoods are given explicitly, which is necessary for the resulting constants to be independent of time. The following technical lemma is needed when verifying the assumptions made in \cite{Hol70}. The proof requires the discussion of a linear quadratic optimal control problem which can be found in Appendix \ref{Appendix A}.
\begin{lemma}\label{lem: DPhi}
	Denote $h_0  = (\Tilde{x},0,0,0)\in X_t$. Further let $h = (x,v,p,\omega) \in X_t$ be fixed but arbitrary. Then $D\Phi_t(h_0) \colon X_t \rightarrow Y_t$ is bijective. Further there exist constants $\Tilde{M}$, $\breve{M} > 0$ independent of $t$ such that
	\begin{equation*}
		\begin{aligned}
			\Vert D\Phi_t(h_0) - D\Phi_t(h) \Vert_{L(X_t,Y_t)}
			&\leq \Tilde{M} \Vert h_0 - h \Vert_{X_t} \text{ and }\\
			\Vert D\Phi_t(h_0)^{-1} \Vert_{L(Y_t,X_t)} 
			&\leq \breve{M}.
		\end{aligned}
	\end{equation*}
\end{lemma}
\begin{proof}
	First note that for $h = (x,v,p,\omega) \in X_t$ and $w = (\mu,f,a,-l_1,b,-l_2) \in Y_t$ it holds $D\Phi_t(\Tilde{x},0,0,0)(x,v,p,\omega) = (\mu,f,a,-l_1,b,-l_2) $ if and only if
	\begin{equation}\label{eqSys}
		\left\{\begin{array}{ccl}
			\omega &=&\mu,\\
			\dot{x} &=& Ax + G(\Tilde{x} \otimes I_n)x + G(I_n \otimes \Tilde{x})x + Fv +f,\\
			x(t) &=& a,\\
			\dot{p} &=& -A^\top p -(\Tilde{x}^\top \otimes I_n) G^\top p - (I_n \otimes \Tilde{x}^\top)G^\top p + \alpha C^\top (\omega - C x) - l_1,\\
			p(0) &=& b - x(0),\\
			v + F^\top p + l_2 &=& 0.
		\end{array}
		\right.
	\end{equation}
	We now employ results presented in Appendix \ref{Appendix A}. System \eqref{eqSys} is equivalent to $\omega = \mu$ and $(x,v,p)$ being the solution of \eqref{Appendix KKT system} associated with $\breve{x}_t = \tilde{x}\vert_{[0,t]}$, $\rho = 0$ and the data given by $w \in Y_t$. Since $\Vert \tilde{x}\vert_{[0,t]} \Vert_{L^\infty(0,t;\mathbb{R}^n)} \leq \Vert \tilde{x} \Vert_{L^\infty(0,t;\mathbb{R}^n)}$ and $0 < \beta$, Proposition \ref{Proposition: Appendix A} yields the unique solvability of the system and hence shows that $D\Phi_t(h_0)$ is bijective. The estimate presented in Proposition \ref{Proposition: Appendix A} together with the definition of the operator norm further shows the bound for $\Vert D\Phi_t(h_0) \Vert_{L(Y_t,X_t)}$. To show the first estimate note that for any $(z,u,q,\eta) \in X_t$ satisfying $ \Vert (z,u,q,\eta) \Vert_{X_t} = 1 $ it holds
	\begin{equation*}
		\begin{aligned}
			&\Vert D\Phi_t(h_0)(z,u,q,\eta) - D\Phi_t(h)(z,u,q,\eta) \Vert_{Y_t}\\
			&=  \left\Vert \begin{matrix}
				0 \\  
				G( (x-\Tilde{x}) \otimes z) + G( z \otimes (x - \Tilde{x}) )\\
				0 \\ 
				\left( ( (\Tilde{x} - x)^\top \otimes I_n ) + ( I_n \otimes (\Tilde{x} - x)^\top ) \right)G^\top q - \left( (z^\top \otimes I_n) + (I_n \otimes z^\top) \right)G^\top p \\
				0 \\ 0
			\end{matrix} \right\Vert_{Y_t}\\
			&\leq 2 \Vert G \Vert_2 \max \left(
			\Vert z \Vert_{V_t} \Vert x - \Tilde{x} \Vert_{V_t},
			\Vert q \Vert_{V_t} \Vert x - \Tilde{x} \Vert_{V_t} 
			+ \Vert z \Vert_{V_t} \Vert p \Vert_{V_t} \right) \\
			&\leq 2 \Vert G \Vert_2 \left( \Vert x - \Tilde{x} \Vert_{V_t} + \Vert p \Vert_{V_t} \right)
			\leq 4 \Vert G \Vert_2 \Vert h - h_0 \Vert_{X_t}.
		\end{aligned}
	\end{equation*}
	Taking the supremum over all such $(z,u,q,\eta) \in X_t$ yields the existence of $\tilde{M}$.
\end{proof}
It follows the proof that $\Phi_t$ has a local inverse, with a domain and image space that are independent of time.
\begin{lemma}\label{Lemma: Four Cinfty Mappings}
	Let $t \in (0,T]$ be arbitrary. Then there exist constants $\delta_3 >0$ and $\delta_3^\prime >0$ independent of $t$ and three $C^\infty$ functions $\mathcal{X}_t$, $\mathcal{U}_t$ and $\mathcal{P}_t$ such that for all $(\xi,\omega) \in \mathcal{U}_{\delta_3}(0)$ the quadruple $\left( \mathcal{X}_t(\xi,\omega), \mathcal{U}_t(\xi,\omega), \mathcal{P}_t(\xi,\omega), \omega \right)$ is the unique solution to 
	\begin{equation}
		\Phi_t(x,v,p,\omega) = (\omega, 0,\xi,0,0,0),~~~(x,v,p,\omega) \in \mathcal{U}_{\delta_3^\prime}(\tilde{x},0,0,0).
	\end{equation}
\end{lemma}
\begin{proof}
	Let $t \in (0,T]$ be fixed. The proof will be carried out by an application of the implicit function theorem presented in \cite{Hol70} to the $C^\infty$-mapping $P_t(w,h) = \Phi_t(h) - w$. With $w_0 = 0 \in Y_t$ and $h_0 = (\Tilde{x},0,0,0) \in X_t$ it holds $P_t(w_0,h_0) = 0$. According to Lemma \ref{lem: DPhi} $D_h P_t = D \Phi_t$ exists, is continuous in $Y_t \times X_t$ and has a bounded linear inverse. Note that the upper bound of the norm of the inverse $\breve{M}$ is independent of $t$. With the non-decreasing functions $g_1(s,i) = \Tilde{M} i$ and $g_2(s) = s$ and the constants $\alpha = \tfrac{1}{2}$, $\delta_3 = \tfrac{1}{4\Tilde{M} \Breve{M}^2}$ and $\delta_3^\prime = \tfrac{1}{2\Tilde{M}\Breve{M}}$ all assumptions are fulfilled and the assertion holds for the chosen, time-independent constants $\delta_3$ and $\delta_3^\prime$. Here the functions $g_1$ and $g_2$ and the parameter $\alpha$ correspond to the functions and constant with the same name used in \cite{Hol70}. The constants $\delta_3$ and $\delta_3^\prime$ correspond to $\delta$ and $\epsilon$ in \cite{Hol70} respectively.
	
	At this point let us note that the implicit function theorem presented in \cite{Hol70} only yields continuity of the three functions, but not differentiability. However, comparing the assumptions and the proof with the ones of more classical versions (see for example \cite[Theorem 4.E]{Zei95AMS109}) shows that the neighborhoods are only adjusted within the application of the fixed point argument. The part of the proof that shows that differentiability of order $k$ of $P_t$ carries over to $\mathcal{X}_t$, $\mathcal{U}_t$ and $\mathcal{P}_t$ does not alter the neighborhoods at all. Hence the assertion of higher regularity holds for the chosen constants. 
\end{proof}
Note that this result does not yet imply that the pair $(\mathcal{X}_t,\mathcal{U}_t)$ minimizes the cost functional. It has only been shown that the triple  $\left( \mathcal{X}_t(\xi,\omega), \mathcal{U}_t(\xi,\omega), \mathcal{P}_t(\xi,\omega) \right)$ satisfies the first order optimality condition. The following proposition ensures that it actually solves the optimal control problem.
\begin{proposition}\label{Proposition: Mappings Data to Minimizer}
	Let $t \in (0,T]$. Then there exists $\delta_4 \in (0,\min(\delta_2, \delta_3, \delta_3^\prime)]$ independent of $t$ such that for all $(\xi,\omega) \in \mathcal{U}_{\delta_4}(0)$ the optimal control problem \eqref{Control Problem} admits exactly one solution. It is given by $(\bar{x},\bar{v})=\left( \mathcal{X}_t(\xi,\omega), \mathcal{U}_t(\xi,\omega) \right).$ The associated adjoint state is given by $\mathcal{P}_t(\xi,\omega)$.
\end{proposition}
\begin{proof}
	For now set $\delta_4 = \min(\delta_2, \delta_3)$. According to Proposition \ref{Proposition: Existence} there exists a minimizer $(\bar{x},\bar{v})$ and it holds
	\begin{equation*}
		\max\left( \Vert \bar{x} - \Tilde{x} \Vert_{V_t}, \Vert \bar{v} \Vert_{L^2(0,t;\mathbb{R}^m)} \right)
		\leq M_1 \max\left( \Vert \xi \Vert, \Vert \omega \Vert_{L^2(0,t;\mathbb{R}^r)} \right) < M_1 \delta_4.
	\end{equation*}
	Furthermore, from Proposition \ref{Proposition: FirstOrderOptimalityCondition} we obtain existence of the associated adjoint state $p$ which satisfies 
	\begin{equation*}
		\Vert p \Vert_{V_t}
		\leq M_2 \max\left( \Vert \xi \Vert, \Vert \omega \Vert_{L^2(0,t;\mathbb{R}^r)} \right)
		< M_2 \delta_4.
	\end{equation*}
	After a suitable reduction of $\delta_4$ one obtains 
	\begin{equation*}
		\max\left( \Vert \bar{x} - \Tilde{x} \Vert_{V_t}, \Vert \bar{v} \Vert_{L^2(0,t;\mathbb{R}^m)}, \Vert p \Vert_{V_t}, \Vert \omega \Vert_{L^2(0,t;\mathbb{R}^r)} \right)
		< \delta_3^\prime.
	\end{equation*}
	Since $(\bar{x},\bar{v},p)$ is an optimal triple for the control problem (\ref{Control Problem}) given by $ \omega$ and $\xi$, it holds $\Phi_t(\bar{x},\bar{v},p,\omega) = (\omega,0,\xi,0,0,0)$. Because $(\xi,\omega) \in \mathcal{U}_{\delta_4}(0)$, Lemma \ref{Lemma: Four Cinfty Mappings} yields the assertion.
\end{proof}
With the smoothness of the cost functional an application of the chain rule immediately yields regularity of the value function with respect to $\xi$ and $\omega$.
\begin{corollary}\label{Corollary: Value function C infty wrt space}
	For any fixed $t \in (0,T]$ the value function $\mathcal{V}(t,\cdot,\cdot) \colon \mathcal{U}_{\delta_4}(0) \rightarrow \mathbb{R}$ associated with the optimal control problem \eqref{Control Problem}  is of class $C^\infty$.
\end{corollary}
We conclude this subsection by a representation of the optimal control evaluated in the final time in terms of the gradient of the value function. Since the proof is done using well-know techniques, see \cite{CanF91}, it is omitted here.
\begin{corollary}\label{Corollary: Feedback Representation Optimal Control}
	Let $t \in (0,T]$ and $(\xi,\omega) \in \mathcal{U}_{\tfrac{1}{2}\delta_4}(0)$. Then the adjoint $ p(t) = \mathcal{P}_t(\xi,\omega)(t)$ and the minimizing control $\bar{v}(t) = \mathcal{U}_t(\xi,\omega)(t)$ evaluated in the final time $t$ are characterized by
	\begin{equation*}
		p(t) = - \nabla_\xi \mathcal{V} (t, \xi, \omega) \text{  and  }
		\bar{v}(t) = F^\top \nabla_\xi \mathcal{V}(t,\xi,\omega).
	\end{equation*}
\end{corollary}

\subsection{Regularity in time}
This subsection aims at showing time regularity of the value function. One of the essential steps is to prove time uniform bounds for the partial derivatives of the value function $\mathcal{V}$. This will be done by a characterization of the triples of partial derivatives of $\mathcal{X}_t$, $\mathcal{U}_t$ and $\mathcal{P}_t$ as solutions of linear quadratic optimal control problems which is presented in Appendix \ref{Appendix A}.
\begin{proposition}\label{Proposition: Bounds for derivatives of the value function}
	Let $t \in (0,T]$ and $(\xi, \omega) \in \mathcal{U}_{\delta_4}(0)$. Let $\mu \in L^2(0,t;\mathbb{R}^r)$ and $z_1,z_2,z_3,z_4 \in \mathbb{R}^n$ and for $i = 1,...,4$ denote $z^i = (z_1,...,z_i)$. Then there exist $\check{M}_1, \check{M}_2, \check{M}_3, \check{M}_4, \check{M}_5,\check{M}_6 > 0$ independent of $t$ such that for $i=1,...,4$ it holds
	\begin{equation*}
		\begin{aligned}
			\left\vert D^i_{\xi^i}\mathcal{V}(t,\xi,\omega)z^i \right\vert 
			&\leq \check{M}_i \prod_{k = 1}^i \Vert z_k \Vert,\\
			\left\vert D^2_{\omega \xi}\mathcal{V}(t,\xi,\omega)(\mu, z_1) \right\vert 
			&\leq \check{M}_5 \Vert \mu \Vert_{L^2(0,t;\mathbb{R}^r)} \Vert z_1 \Vert,\\
			\left\vert D^3_{\omega \xi^2}\mathcal{V}(t,\xi,\omega)(\mu, z_2,z_1) \right\vert 
			&\leq \check{M}_6 \Vert \mu \Vert_{L^2(0,t;\mathbb{R}^r)} \Vert z_1 \Vert \Vert z_2 \Vert.
		\end{aligned}
	\end{equation*}
\end{proposition}
\begin{proof}
	For any $t \in(0,T]$ and $(\xi,\omega) \in \mathcal{U}_{\delta_4}(0)$ it holds
	\begin{equation*}
		\mathcal{V}(t,\xi,\omega) = J(\mathcal{X}_t(\xi,\omega), \mathcal{U}_t(\xi,\omega);t,\omega).
	\end{equation*}
	Forming the appropriate derivatives of this equation, applying the chain rule to the derivative on the right hand side and utilizing the bounds from Proposition \ref{Proposition: Existence} and Lemma \ref{Lemma: Bounds for derivatives} yields the assertion for $t > 0$. Noting that $\mathcal{V}(0,\xi,\omega) = \frac{1}{2} \Vert \xi \Vert^2$ concludes the proof.
\end{proof}
With these tools at hand it is now possible to show time continuity of the value function. To obtain this result we assume $\omega$ to be essentially bounded.
\begin{proposition}\label{Proposition: Continuity of the Value Function}
	Let $(\xi,\omega) \in \mathcal{U}_{\tfrac{1}{2}\delta_4}(0) \subset \mathbb{R}^n \times L^2(0,T;\mathbb{R}^r)$ and assume $\omega \in L^\infty(0,T;\mathbb{R}^r)$. Then the mapping $t \mapsto \mathcal{V}(t,\xi,\omega)$ is continuous in $[0,T]$.
\end{proposition}
\begin{proof}
	Let $t \in [0,T)$ be arbitrary. Then for sufficiently small $\tau > 0$ it holds $t + \tau \in (0,T)$ and with Bellman's principle it holds
	\begin{equation}\label{RightLimitContinuity}
		\begin{aligned}
			\mathcal{V}(t+\tau,\xi,\omega) - \mathcal{V}(t,\xi,\omega)
			= \mathcal{V}(t, \mathcal{X}_{t+\tau}(\xi,\omega)(t) - \Tilde{x}(t),\omega) - \mathcal{V}(t,\xi,\omega)&\\
			+ \, \frac{1}{2} \int_t^{t+\tau} \left\Vert \mathcal{U}_{t+\tau}(\xi,\omega) \right\Vert^2 + \alpha \left\Vert \omega - C \left( \mathcal{X}_{t+\tau}(\xi,\omega) - \Tilde{x} \right) \right\Vert^2 \,\mathrm{d}s.&
		\end{aligned}
	\end{equation}
	To see that for $\tau \searrow 0$ the first summand converges to zero, note that $\mathcal{X}_{t+\tau}(\xi,\omega) \in C^1([0,t+\tau];\mathbb{R}^n)$. This is due to the fact that the optimal control $\mathcal{U}_{t+\tau}(\xi,\omega) = - F^\top \mathcal{P}_{t+\tau}(\xi,\omega)$ is continuous. It follows
	\begin{equation}\label{instantoCall}
		\begin{aligned}
			\Vert \mathcal{X}_{t+\tau}(\xi,\omega)(t) - \Tilde{x}(t) - \xi \Vert
			&= \left\Vert \mathcal{X}_{t+\tau}(\xi,\omega)(t+\tau) - \int_t^{t+\tau} \dot{\mathcal{X}}_{t+\tau}(\xi,\omega)(s) \,\mathrm{d}s- \Tilde{x}(t) - \xi \right\Vert\\
			&\leq  \Vert \Tilde{x}(t+\tau) - \Tilde{x}(t) \Vert + \tau \Vert \dot{\mathcal{X}}_{t+\tau} (\xi,\omega) \Vert_{L^\infty(0,t+\tau;\mathbb{R}^n)}.
		\end{aligned}
	\end{equation}
	For any $s \in [0,t+\tau]$ it holds
	\begin{equation*}
		\begin{aligned}
			\Vert \dot{\mathcal{X}}_{t+\tau} (\xi,\omega)(s) \Vert
			&\leq \Vert A \mathcal{X}_{t+\tau} (\xi,\omega)(s) \Vert + \Vert G( \mathcal{X}_{t+\tau} (\xi,\omega)(s) \otimes \mathcal{X}_{t+\tau} (\xi,\omega)(s) ) \Vert\\ 
			&+ \Vert F \mathcal{U}_{t+\tau}(\xi,\omega)(s) \Vert\\
			&\leq \Vert A \Vert_2 \Vert \mathcal{X}_{t+\tau} (\xi,\omega) \Vert_{L^\infty(0,t+\tau;\mathbb{R}^n)}
			+ \Vert G \Vert_2 \Vert \mathcal{X}_{t+\tau} (\xi,\omega) \Vert_{L^\infty(0,t+\tau;\mathbb{R}^n)}^2\\
			&+ \Vert F \Vert_2^2 \Vert \mathcal{P}_{t+\tau} (\xi,\omega) \Vert_{L^\infty(0,t+\tau;\mathbb{R}^n)}.
		\end{aligned}
	\end{equation*}
	By Proposition \ref{Proposition: Existence} and Proposition \ref{Proposition: FirstOrderOptimalityCondition} it follows
	\begin{equation*}
		\begin{aligned}
			\Vert \dot{\mathcal{X}}_{t+\tau} (\xi,\omega)(s) \Vert
			&\leq \Vert A \Vert_2 \left( \Vert \tilde{x} \Vert_{L^\infty(0,t+\tau;\mathbb{R}^n)} + M_1 \max \left( \Vert \xi \Vert, \Vert \omega \Vert_{L^2(0,t+\tau,\mathbb{R}^r)} \right) \right)\\
			&+ 2 \Vert G \Vert_2 \left( \Vert \tilde{x} \Vert_{L^\infty(0,t+\tau;\mathbb{R}^n)}^2 + M_1^2 \max \left( \Vert \xi \Vert^2, \Vert \omega \Vert_{L^2(0,t+\tau,\mathbb{R}^r)}^2 \right) \right)\\
			&+ M_2 \Vert F \Vert_2^2  \max \left( \Vert \xi \Vert^2, \Vert \omega \Vert_{L^2(0,t+\tau,\mathbb{R}^r)}^2 \right).
		\end{aligned}
	\end{equation*}
	Since $t+\tau < T$, there exists some $c>0$ independent of $\tau$ such that
	\begin{equation*}
		\Vert \dot{\mathcal{X}}_{t+\tau}(\xi,\omega) \Vert_{L^\infty(0,t+\tau;\mathbb{R}^n)} \leq c
	\end{equation*}
	holds for all sufficiently small $\tau$. Therefore the right hand side of (\ref{instantoCall}) tends to zero for $\tau \searrow 0$. The fact that $\nu \mapsto \mathcal{V}(t,\nu,\omega)$ is continuous in $\nu = \xi$ then implies that the first summand of (\ref{RightLimitContinuity}) goes to zero for $\tau \searrow 0$.\\
	For the second summand it holds
	\begin{equation*}
		\begin{aligned}
			&\int_t^{t+\tau} \left\Vert \mathcal{U}_{t+\tau}(\xi,\omega) \right\Vert^2 + \alpha \left\Vert \omega - C \left( \mathcal{X}_{t+\tau}(\xi,\omega) - \Tilde{x} \right) \right\Vert^2 \,\mathrm{d}s\\
			&\leq \int_t^{t+\tau} \Vert F \Vert_2^2 \Vert \mathcal{P}_{t+\tau}(\xi,\omega) \Vert^2 
			+ 2 \alpha \Vert \omega \Vert_{L^\infty(0,T;\mathbb{R}^r)}^2 
			+ 2 \alpha \Vert C \Vert_2^2 \Vert \mathcal{X}_{t+\tau}(\xi,\omega) - \Tilde{x} \Vert^2 \,\mathrm{d}s\\
			&\leq \tau \left( \Vert F \Vert_2^2 \Vert \mathcal{P}_{t+\tau}(\xi,\omega) \Vert_{L^\infty(0,t+\tau;\mathbb{R}^n)}^2
			+ 2\alpha \Vert \omega \Vert_{L^\infty(0,T;\mathbb{R}^r)}^2 \right.\\
			&+ \left. 2 \alpha \Vert C \Vert_2^2 \Vert \mathcal{X}_{t+\tau}(\xi,\omega) - \Tilde{x} \Vert_{L^\infty(0,t+\tau;\mathbb{R}^n)}^2 \right).
		\end{aligned}
	\end{equation*}
	With Proposition \ref{Proposition: Existence} and Proposition \ref{Proposition: FirstOrderOptimalityCondition} all terms on the right hand side can be estimated from above by a constant independent of $\tau$. Hence the right hand side goes to zero for $\tau \searrow 0$. This shows that for all $t \in [0,T)$ it holds
	\begin{equation}\label{LimitFromTheRight}
		\mathcal{V}(t+\tau,\xi,\omega) \rightarrow \mathcal{V}(t,\xi,\omega),~~~~~\text{for}~~~~~ \tau \searrow 0.
	\end{equation}
	Now let $t \in (0,T]$ be fixed but arbitrary. Then for sufficiently small $\tau > 0$ it holds $t - \tau \in (0,T)$. With Bellman's principle it follows
	\begin{equation}\label{LeftLimitContinuity} 
		\begin{aligned}
			\mathcal{V}(t,\xi,\omega) - \mathcal{V}(t-\tau,\xi,\omega)
			= \mathcal{V}(t-\tau, \mathcal{X}_t(\xi,\omega)(t-\tau) - \Tilde{x}(t-\tau), \omega) - \mathcal{V}(t-\tau,\xi,\omega)&\\
			+ \int_{t-\tau}^t \Vert \mathcal{U}_t(\xi,\omega) \Vert^2 + \alpha \Vert \omega - C(\mathcal{X}_t(\xi,\omega) - \Tilde{x}) \Vert^2 \,\mathrm{d}s.&
		\end{aligned}
	\end{equation}
	To see convergence of the first summand note that
	\begin{equation*}
		\begin{aligned}
			\mathcal{X}_t(\xi,\omega)(t-\tau) - \Tilde{x}(t-\tau)
			&= \mathcal{X}_t(\xi,\omega)(t) - \int_{t-\tau}^t \dot{\mathcal{X}}_t(\xi,\omega)(s) \,\mathrm{d}s - \Tilde{x}(t-\tau)\\
			&= \xi + \Tilde{x}(t) - \Tilde{x}(t-\tau) - \int_{t-\tau}^t \dot{\mathcal{X}}_t(\xi,\omega)(s) \,\mathrm{d}s.
		\end{aligned}
	\end{equation*}
	Since $\mathcal{X}_t(\xi,\omega) \in C^1([0,t];\mathbb{R}^n)$, one has that $\Vert \dot{\mathcal{X}}_{t}(\xi,\omega) \Vert_{L^\infty(0,t;\mathbb{R}^n)}$ is finite. From here on assume that $\tau$ is small enough such that $t-\tau \in (0,T)$ and
	\begin{equation*}
		\Vert \mathcal{X}_t(\xi,\omega)(t-\tau) - \Tilde{x}(t-\tau) \Vert
		\leq \Vert \xi \Vert + \Vert \tilde{x}(t) - \tilde{x}(t-\tau) \Vert + \tau \Vert\dot{\mathcal{X}}(\xi,\omega) \Vert_{L^\infty(0,t;\mathbb{R}^n)}
		< \delta_4.
	\end{equation*}
	With Taylor there exists some $\theta \in [0,1]$ such that
	\begin{equation*}
		\begin{aligned}
			&\mathcal{V}(t-\tau, \mathcal{X}_t(\xi,\omega)(t-\tau) - \Tilde{x}(t-\tau), \omega) - \mathcal{V}(t-\tau,\xi,\omega)\\
			= \left\langle \vphantom{\int_{t-\tau}^t} \right.&\nabla_\xi \mathcal{V} \left(t-\tau, \xi + \theta \left( \Tilde{x}(t) - \Tilde{x}(t-\tau) - \int_{t-\tau}^t \dot{\mathcal{X}}_t(\xi,\omega)(s) \,\mathrm{d}s \right),\omega \right) ,\\
			& \left( \Tilde{x}(t) - \Tilde{x}(t-\tau) - \int_{t-\tau}^t \dot{\mathcal{X}}_t(\xi,\omega)(s) \,\mathrm{d}s \right) \left. \vphantom{\int_{t-\tau}^t} \right\rangle.
		\end{aligned}
	\end{equation*}
	Then Proposition \ref{Proposition: Bounds for derivatives of the value function} yields
	\begin{equation*}
		\begin{aligned}
			&\left\vert \mathcal{V}(t-\tau, \mathcal{X}_t(\xi,\omega)(t-\tau) - \Tilde{x}(t-\tau), \omega) - \mathcal{V}(t-\tau,\xi,\omega) \right\vert\\
			&\leq \check{M}_1 \left\Vert \Tilde{x}(t) - \Tilde{x}(t-\tau) - \int_{t-\tau}^t \dot{\mathcal{X}}_t(\xi,\omega)(s) \,\mathrm{d}s \right\Vert \\
			&\leq \check{M}_1 \Vert \Tilde{x}(t) - \Tilde{x}(t-\tau) \Vert + \check{M}_1 \tau \Vert \dot{\mathcal{X}}_t(\xi,\omega) \Vert_{L^\infty(0,t;\mathbb{R}^n)}.
		\end{aligned}
	\end{equation*}
	Therefore the first summand in (\ref{LeftLimitContinuity}) tends to zero for $\tau \searrow 0$. The second summand can be treated with the same arguments as the second summand of (\ref{RightLimitContinuity}). This finally shows that for all $t \in (0,T]$ it holds
	\begin{equation*}
		\mathcal{V}(t-\tau,\xi,\omega) \rightarrow \mathcal{V}(t,\xi,\omega),~~~~~\text{for}~~~~~ \tau \searrow 0
	\end{equation*}
	and the assertion is shown.
\end{proof}
Because all constants are independent of time the continuity of the value function carries over to its spatial derivatives.
\begin{proposition}\label{Proposition: Space derivatives continuous in time}
	Let $(\xi,\omega) \in \mathcal{U}_{\tfrac{1}{2}\delta_4}(0) \subset \mathbb{R}^n \times L^2(0,T;\mathbb{R}^r)$ and assume $\omega \in L^\infty(0,T;\mathbb{R}^r)$. Then the mappings $ t \mapsto \nabla_\xi \mathcal{V}(t,\xi,\omega)$, $ t \mapsto \nabla^2_{\xi\xi} \mathcal{V}(t,\xi,\omega)$ and $ t \mapsto \nabla^3_{\xi^3} \mathcal{V}(t,\xi,\omega)$ are continuous in $[0,T]$.
\end{proposition}
\begin{proof}
	First, let us show that for any $i \in \{ 1,...,n \}$ the partial derivative $t \mapsto \frac{\partial}{\partial \xi_i} \mathcal{V}(t,\xi,\omega)$ has the claimed regularity. Note that for any $h$ with $|h|$ small enough, we have $ \Vert \xi + h e_i \Vert < \tfrac{1}{2} \delta_4 $. Then Proposition \ref{Proposition: Continuity of the Value Function} shows that for such $h$ it holds $f_h \in C([0,T])$, where
	\begin{equation*}
		f_h(t) \coloneqq \frac{1}{h} \left( \mathcal{V}(t, \xi + h e_i,\omega) - \mathcal{V}(t,\xi,\omega) \right).
	\end{equation*}
	By definition $f_h$ converges pointwise to $\frac{\partial}{\partial \xi_i} \mathcal{V}(\cdot,\xi,\omega)$ for $h \to 0$. Let us show that $f_h$ is a Cauchy sequence in $C([0,T])$. The fact that $C([0,T])$ is a complete space then implies the uniform convergence of $f_h$ and the pointwise limit coincides with the uniform limit, which is an element of $C([0,T])$.
	
	Specifically we need to show that for any $\epsilon > 0$ there exists $h_0 > 0$ such that for all $h_1$, $h_2$ with $\vert h_1 \vert$, $\vert h_2 \vert < h_0$ it follows
	\begin{equation*}
		\Vert f_{h_1} - f_{h_2} \Vert_{C([0,T])} < \epsilon.
	\end{equation*}
	Taylor's Theorem yields existence of $\theta_1$, $\theta_2 \in [0,1]$ such that
	\begin{equation*}
		\begin{aligned}
			&\sup_{s\in[0,T]} \left\vert
			\frac{1}{h_1} \left( \mathcal{V}(s,\xi+h_1 e_i,\omega) - \mathcal{V}(s,\xi,\omega)  \right)
			- \frac{1}{h_2} \left( \mathcal{V}(s,\xi+h_2 e_i,\omega) - \mathcal{V}(s,\xi,\omega)  \right)
			\right\vert \\
			&=\sup_{s\in[0,T]} \left\vert
			\frac{1}{h_1} \nabla_\xi \mathcal{V}(s,\xi + \theta_1 h_1 e_i,\omega)h_1 e_i
			- \frac{1}{h_2} \nabla_\xi \mathcal{V}(s,\xi + \theta_2 h_2 e_i,\omega)h_2 e_i
			\right\vert \\
			&=\sup_{s\in[0,T]} \left\vert
			\frac{\partial}{\partial \xi_i} \mathcal{V}(s,\xi + \theta_1 h_1 e_i,\omega)
			- \frac{\partial}{\partial \xi_i} \mathcal{V}(s,\xi + \theta_2 h_2 e_i,\omega)
			\right\vert \\
		\end{aligned}
	\end{equation*}
	After another application of Taylor's Theorem it follows the existence of $\theta_3 \in [0,1]$ such that
	\begin{equation*}
		\begin{aligned}
			\Vert f_{h_1} - f_{h_2} \Vert_{C([0,T])}
			&= \sup_{s\in[0,T]} \left\vert
			\nabla_\xi \frac{\partial}{\partial \xi_i} \mathcal{V}(s, \xi + \theta_2 h_2 e_i + \theta_3 (\theta_1 h_1 - \theta_2 h_2) e_i ,\omega) (\theta_1 h_1 - \theta_2 h_2) e_i
			\right\vert \\
			&= (\theta_1 h_1 - \theta_2 h_2) \sup_{s\in[0,T]} \left\vert
			\frac{\partial^2}{\partial \xi_i^2} \mathcal{V} (s, \xi + (\theta_2 h_2 + \theta_1 \theta_3 h_1 - \theta_2 \theta_3 h_2) e_i ,\omega) 
			\right\vert
		\end{aligned}
	\end{equation*}
	A sufficiently small upper bound $h_0$ for $\vert h_1 \vert$ and $\vert h_2 \vert$ yields
	\begin{equation*}
		\left \Vert \xi + (\theta_2 h_2 + \theta_1 \theta_3 h_1 - \theta_2 \theta_3 h_2) e_i \right\Vert
		\leq \Vert \xi \Vert + 2 \vert h_2 \vert + \vert h_1 \vert 
		< \tfrac{1}{2} \delta_4.
	\end{equation*}
	With the bound from Proposition \ref{Proposition: Bounds for derivatives of the value function} and a possible further decrease of $h_0$ it follows
	\begin{equation*}
		\Vert f_{h_1} - f_{h_2} \Vert_{C([0,T])}
		\leq \left( \vert h_1 \vert + \vert h_2 \vert \right) \check{M}_2
		< \epsilon.
	\end{equation*}
	This concludes the proof that $f_h$ is a Cauchy sequence in a complete space. Hence the pointwise limit $\mathcal{V}(\cdot,\xi,\omega)$ is also the uniform limit and an element of $C([0,T])$. 
	
	Since Proposition \ref{Proposition: Bounds for derivatives of the value function} includes bounds for the spatial derivatives of $\mathcal{V}$ up to order four, the assertion for derivatives of order two and three can be shown analogously.
\end{proof}
We can finally show that for sufficiently small $(\xi,\omega)$ the value function is differentiable in time and  its derivative is characterized by the HJB equation. For some of the technical proofs, we refer to Appendix \ref{Appendix C}.
\begin{theorem}\label{Theorem: HJB holds}
	Let $(\xi,\omega) \in \mathcal{U}_{\tfrac{1}{2}\delta_4}(0) \subset \mathbb{R}^n \times L^2(0,T;\mathbb{R}^r)$ and assume $\omega$ to be continuous. Then the mapping $ t \mapsto \mathcal{V}(t,\xi,\omega) $ is differentiable in $(0,T]$. Its derivative in $t \in (0,T]$ is given by
	\begin{equation}\label{HJB rigorous}
		\partial_t \mathcal{V}(t,\xi,\omega) 
		= - \left\langle \nabla_\xi \mathcal{V}(t,\xi, \omega) , h(t,\xi) \right\rangle - \frac{1}{2} \left\Vert F^\top \nabla_\xi \mathcal{V}(t,\xi,\omega) \right\Vert^2 + \frac{\alpha}{2} \left\Vert \omega(t) - C \xi \right\Vert^2, 
	\end{equation}
	where $h(t,\xi) = A \xi + G(\Tilde{x}(t) \otimes \xi) + G(\xi \otimes \Tilde{x}(t)) + G(\xi \otimes \xi)$.
\end{theorem}
\begin{proof}
	For sufficiently small $\tau > 0$ it holds
	\begin{equation*}
		\begin{aligned}
			\mathcal{X}_{t+\tau}(\xi,\omega)(t) - \tilde{x}(t)
			&= \mathcal{X}_{t+\tau}(\xi,\omega) (t+\tau) - \int_t^{t+\tau} \dot{\mathcal{X}}_{t+\tau}(\xi,\omega) \,\mathrm{d}s - \tilde{x}(t)\\
			&= \xi + \tilde{x}(t+\tau) - \tilde{x}(t) - \int_t^{t+\tau} \dot{\mathcal{X}}_{t+\tau}(\xi,\omega) \,\mathrm{d}s.
		\end{aligned}
	\end{equation*}
	Hence with Bellman's principle it follows
	\begin{equation*}
		\begin{aligned}
			&\frac{1}{\tau} \left( \mathcal{V}(t+\tau,\xi,\omega) - \mathcal{V}(t,\xi,\omega) \right)\\
			&= \frac{1}{\tau} \mathcal{V}\left(t, \xi + \tilde{x}(t+\tau) - \tilde{x}(t) - \int_t^{t+\tau} \dot{\mathcal{X}}_{t+\tau}(\xi,\omega) \,\mathrm{d}s, \omega \right) - \frac{1}{\tau} \mathcal{V}(t,\xi,\omega)  \\
			&+ \frac{1}{2 \tau} \int_t^{t+\tau} \Vert \mathcal{U}_{t+\tau}(\xi,\omega) \Vert^2 + \alpha \Vert \omega - C(\mathcal{X}_{t+\tau}(\xi,\omega) - \tilde{x}) \Vert^2 \,\mathrm{d}s.
		\end{aligned}
	\end{equation*}
	With Lemma \ref{Lemma: V differentiable Auxiliary4} for $\tau \searrow 0 $ the second term converges to 
	\begin{equation*}
		\frac{1}{2} \Vert \mathcal{U}_t(\xi,\omega)(t) \Vert^2 + \frac{\alpha}{2} \Vert \omega(t) - C(\mathcal{X}_t(\xi,\omega)(t) - \tilde{x}(t))\Vert^2.
	\end{equation*}
	A Taylor expansion shows that there exists some $\theta(\tau) \in [0,1]$ such that the first term is equal to 
	\begin{equation*}
		\begin{aligned}
			&\frac{1}{\tau} \left\langle \vphantom{\int_t^{t+\tau}} \right. \nabla_\xi \mathcal{V}\left(t, \xi + \theta(\tau) \left( \tilde{x}(t+\tau) - \tilde{x}(t) - \int_t^{t+\tau} \dot{\mathcal{X}}_{t+\tau}(\xi,\omega) \,\mathrm{d}s \right), \omega\right) ,\\ 
			&\tilde{x}(t+\tau) - \tilde{x}(t) - \int_t^{t+\tau} \dot{\mathcal{X}}_{t+\tau}(\xi,\omega) \,\mathrm{d}s \left. \vphantom{\int_t^{t+\tau}} \right\rangle\\
			&= \left\langle \vphantom{\int_t^{t+\tau}} \right. \nabla_\xi \mathcal{V}\left(t, \xi + \theta(\tau) \left( \tilde{x}(t+\tau) - \tilde{x}(t) - \int_t^{t+\tau} \dot{\mathcal{X}}_{t+\tau}(\xi,\omega) \,\mathrm{d}s \right), \omega\right) , \\
			&\frac{\tilde{x}(t+\tau) - \tilde{x}(t)}{\tau} - \frac{1}{\tau} \int_t^{t+\tau} \dot{\mathcal{X}}_{t+\tau}(\xi,\omega) \,\mathrm{d}s \left. \vphantom{\int_t^{t+\tau}} \right\rangle.
		\end{aligned}
	\end{equation*}
	Due to Lemma \ref{Lemma: V differentiable Auxiliary3} and the continuity of $\nabla_\xi \mathcal{V}$ with respect to $\xi$ (shown in Corollary \ref{Corollary: Value function C infty wrt space}) the right hand side converges to \begin{equation*}
		\left\langle \nabla_\xi \mathcal{V}(t,\xi,\omega) , \dot{\tilde{x}}(t) - \dot{\mathcal{X}}_t(\xi,\omega)(t) \right\rangle.
	\end{equation*}
	It remains to show convergence of 
	$\frac{1}{\tau} \left(\mathcal{V}(t,\xi,\omega) - \mathcal{V}(t-\tau,\xi,\omega) \right)$.
	First note that it holds
	\begin{equation*}
	\begin{aligned}
		\mathcal{X}_{t}(\xi,\omega)(t-\tau) - \tilde{x}(t-\tau) 
		&= \mathcal{X}_t(\xi,\omega)(t) - \tilde{x}(t-\tau) - \int_{t-\tau}^t \dot{\mathcal{X}}_t(\xi,\omega)(s) \,\mathrm{d}s \\
		&= \xi + \tilde{x}(t) - \tilde{x}(t-\tau) - \int_{t-\tau}^t \dot{\mathcal{X}}_t(\xi,\omega)(s) \,\mathrm{d}s.
	\end{aligned}
	\end{equation*}
	With Bellman's principle it follows
	\begin{equation}\label{Call Once}
		\begin{aligned}
			\frac{1}{\tau} \left(\mathcal{V}(t,\xi,\omega) - \mathcal{V}(t-\tau,\xi,\omega) \right)
			&= \frac{1}{\tau} \left( \mathcal{V}\left(t-\tau,  \mathcal{X}_{t}(\xi,\omega)(t-\tau) - \tilde{x}(t-\tau), \omega \right) - \mathcal{V}(t-\tau,\xi,\omega) \right)\\
			&+ \frac{1}{2 \tau} \int_{t-\tau}^t \Vert \mathcal{U}_t(\xi,\omega) \Vert^2 + \alpha \Vert \omega - C (\mathcal{X}_t(\xi,\omega) - \tilde{x}) \Vert^2 \,\mathrm{d}s.
		\end{aligned}
	\end{equation}
	A second order Taylor expansion shows that the first term of the right hand side is equal to
	\begin{equation*}
		\begin{aligned}
			&= \frac{1}{\tau} \mathcal{V} \left(t-\tau, \xi + \tilde{x}(t) - \tilde{x}(t-\tau) - \int_{t-\tau}^t \dot{\mathcal{X}}_t(\xi,\omega)(s) \,\mathrm{d}s, \omega  \right) - \frac{1}{\tau} \mathcal{V}(t-\tau,\xi,\omega)\\
			&= \frac{1}{\tau} \left\langle \nabla_\xi \mathcal{V}(t-\tau, \xi,\omega) , \tilde{x}(t) - \tilde{x}(t-\tau) - \int_{t-\tau}^t \dot{\mathcal{X}}_t(\xi,\omega)(s) \,\mathrm{d}s \right\rangle\\
			&+ \frac{1}{\tau} \int_0^1 (1-\theta) \left( \tilde{x}(t) - \tilde{x}(t-\tau) - \int_{t-\tau}^t \dot{\mathcal{X}}_t(\xi,\omega)(s) \,\mathrm{d}s \right)^\top\\ 
			&\nabla^2_{\xi\xi} \mathcal{V} \left(t-\tau, \xi + \theta \left( \tilde{x}(t) - \tilde{x}(t-\tau) - \int_{t-\tau}^t \dot{\mathcal{X}}_t(\xi,\omega)(s) \,\mathrm{d}s \right), \omega \right)\\
			&\left( \tilde{x}(t) - \tilde{x}(t-\tau) - \int_{t-\tau}^t \dot{\mathcal{X}}_t(\xi,\omega)(s) \,\mathrm{d}s \right) \,\mathrm{d}\theta.\\
		\end{aligned}
	\end{equation*}
	Since $s \mapsto \nabla_\xi \mathcal{V}(s,\xi,\omega)$ and $s \mapsto \dot{\mathcal{X}}_t(\xi,\omega)(s)$ are continuous in $s = t$, for $\tau \searrow 0$ the first order term converges to 
	\begin{equation*}
		\left\langle \nabla_\xi \mathcal{V}(t,\xi,\omega) , \dot{\tilde{x}}(t) - \dot{\mathcal{X}}_t(\xi,\omega)(t) \right\rangle.
	\end{equation*}
	For sufficiently small $\tau$ Proposition \ref{Proposition: Bounds for derivatives of the value function} yields that the absolute value of the second order term is bounded from above by
	\begin{equation*}
		\tau \int_0^1 \check{M}_2\left\Vert \frac{\tilde{x}(t) - \tilde{x}(t-\tau)}{\tau} - \frac{1}{\tau} \int_{t-\tau}^t \dot{\mathcal{X}}_t(\xi,\omega)(s) \,\mathrm{d}s \right\Vert^2 \,\mathrm{d}\theta.
	\end{equation*}
	Since the term in the integral converges for $\tau \searrow 0$, it is bounded. Therefore the right hand side converges to zero for $\tau \searrow 0$. Since $\mathcal{U}_t(\xi,\omega)(s)$, $\mathcal{X}_t(\xi,\omega)(s)$, $\omega(s)$ and $\tilde{x}(s)$ are continuous in $s = t$, the second term in the right hand side of (\ref{Call Once}) converges to
	\begin{equation*}
		\frac{1}{2} \Vert \mathcal{U}_t(\xi,\omega)(t) \Vert^2 + \frac{\alpha}{2} \Vert \omega(t) - C(\mathcal{X}_t(\xi,\omega)(t) - \tilde{x}(t)) \Vert^2,
	\end{equation*}
	for $\tau \searrow 0$. This finally proves
	\begin{equation*}
		\begin{aligned}
			&\partial_t \mathcal{V}(t,\xi,\omega) 
			= \lim_{\tau \to 0} \frac{\mathcal{V}(t+\tau,\xi,\omega) - \mathcal{V}(t,\xi,\omega)}{\tau} \\
			&= \left\langle \nabla_\xi \mathcal{V}(t,\xi,\omega) , \dot{\tilde{x}}(t) - \dot{\mathcal{X}}_t(\xi,\omega)(t) \right\rangle
			+ \frac{1}{2} \Vert \mathcal{U}_t(\xi,\omega) (t) \Vert^2 + \frac{\alpha}{2} \Vert \omega(t) - C(\mathcal{X}_t(\xi,\omega)(t) - \tilde{x}(t)) \Vert^2\\
			&= - \left\langle \nabla_\xi \mathcal{V}(t,\xi,\omega) , h(t,\xi) \right\rangle - \left\langle \nabla_\xi \mathcal{V}(t,\xi,\omega), F F^\top \nabla_\xi \mathcal{V}(t,\xi,\omega) \right\rangle \\
			&+ \frac{1}{2} \Vert F^\top \nabla_\xi \mathcal{V}(t,\xi,\omega) \Vert^2 + \frac{\alpha}{2} \Vert \omega(t) - C \xi \Vert^2,
		\end{aligned}
	\end{equation*}
	where we used Corollary \ref{Corollary: Feedback Representation Optimal Control}.
	
	This shows the assertion for all $t \in (0,T)$. The result can be extended to $t = T$ by a consideration of this entire work on the time interval $[0,\hat{T}]$ for some $\hat{T} > T$. This argument is repeated implicitly whenever we extend regularity results to the right hand boundary of the time interval. 
\end{proof}
%%%%%%%%%%%%%%%%%%%%%%%%%%%%%%%%%%%%%%%%%%%%%%%%%%%%%%%%%%%%%%%%%%%%%%%%%%%%%
\section{Regularity of the Hessian}\label{Section: Regularity of the Hessian}
%%%%%%%%%%%%%%%%%%%%%%%%%%%%%%%%%%%%%%%%%%%%%%%%%%%%%%%%%%%%%%%%%%%%%%%%%%%%%
In this section it will be shown that for sufficiently small data the Hessian of the value function is an invertible matrix at any time $t$.  We begin with a characterization of the Hessian $\nabla^2_{\xi\xi} \mathcal{V}(t,0,0)$.
\begin{proposition}\label{Proposition: DRE in zero without IV}
	For all $t \in(0,T]$ it holds 
	\begin{equation*}
	\begin{aligned}
		\partial_t \nabla^2_{\xi\xi}\mathcal{V}(t,0,0) 
		&=  \nabla^2_{\xi\xi}\mathcal{V}(t,0,0) B(t) + B^\top(t) \nabla^2_{\xi\xi}\mathcal{V}(t,0,0) \\
		&- \nabla^2_{\xi\xi}\mathcal{V}(t,0,0) FF^\top \nabla^2_{\xi\xi}\mathcal{V}(t,0,0) + \alpha C^\top C,
	\end{aligned}
	\end{equation*}
	where $B(t) \coloneqq - A - G(\Tilde{x}(t) \otimes I_n) - G(I_n \otimes \Tilde{x}(t))$.
\end{proposition}
\begin{proof}
	Since $\omega =0$ is continuous, \eqref{HJB rigorous} holds with $(\xi,\omega) = (\xi,0)$ with $\Vert \xi \Vert$ small enough. Then taking partial derivatives with respect to $\xi_i$ and $\xi_j$ and evaluating in $(\xi,\omega) = (0,0)$ yields
	\begin{equation*}
		\begin{aligned}
			\frac{\partial^2}{\partial \xi_j \partial \xi_i} \partial_t \mathcal{V}(t,0,0)
			&= - \left\langle \frac{\partial^2}{\partial \xi_j \partial \xi_i} \nabla_\xi \mathcal{V}(t,0,0) , h(t,0) \right\rangle
			- \left\langle \frac{\partial}{\partial \xi_i} \nabla_\xi \mathcal{V}(t,0,0) , \frac{\partial}{\partial \xi_j} h(t,0) \right\rangle\\
			&- \left\langle \frac{\partial}{\partial \xi_j} \nabla_\xi \mathcal{V}(t,0,0) , \frac{\partial}{\partial \xi_i} h(t,0) \right\rangle
			- \left\langle \nabla_\xi \mathcal{V}(t,0,0) , \frac{\partial^2}{\partial \xi_j \partial \xi_i} h(t,0) \right\rangle\\
			&- \left\langle F^\top \frac{\partial^2}{\partial \xi_j \partial \xi_i} \nabla_\xi \mathcal{V}(t,0,0) , F^\top \nabla_\xi \mathcal{V}(t,0,0) \right\rangle\\
			&- \left\langle F^\top \frac{\partial}{\partial \xi_i} \nabla_\xi \mathcal{V}(t,0,0) , F^\top \frac{\partial}{\partial \xi_j} \nabla_\xi \mathcal{V}(t,0,0) \right\rangle
			+ \alpha \left\langle C^\top C e_i , e_j \right\rangle.
		\end{aligned}
	\end{equation*}
	Note that for any $\eta \in \mathbb{R}^n$ it holds $\mathcal{V}(t,\eta,0) \geq 0$ and $\mathcal{V}(t,0,0) =0$. It follows that $\eta = 0$ is a minimizer of $\mathcal{V}(t,\cdot,0)$ and therefore it holds $\nabla_\xi \mathcal{V}(t,0,0) = 0$ as well as $h(t,0) = 0$. Consequently, we obtain
	\begin{equation*}
		\begin{aligned}
			\frac{\partial^2}{\partial \xi_j \partial \xi_i} \partial_t \mathcal{V}(t,0,0)
			&= - \left\langle \frac{\partial}{\partial \xi_i} \nabla_\xi \mathcal{V}(t,0,0) , \frac{\partial}{\partial \xi_j} h(t,0) \right\rangle
			- \left\langle \frac{\partial}{\partial \xi_j} \nabla_\xi \mathcal{V}(t,0,0) , \frac{\partial}{\partial \xi_i} h(t,0) \right\rangle\\
			&- \left\langle F^\top \frac{\partial}{\partial \xi_i} \nabla_\xi \mathcal{V}(t,0,0) , F^\top \frac{\partial}{\partial \xi_j} \nabla_\xi \mathcal{V}(t,0,0) \right\rangle
			+ \alpha \left\langle C^\top C e_i , e_j \right\rangle.
		\end{aligned}
	\end{equation*}
	With Lemma \ref{Lemma: Switching Derivatives} we can switch the order of differentiation on the left hand side. Writing the resulting $n^2$ equations as a matrix equation yields the assertion.
\end{proof}
Exploiting the fact that the Hessian is continuous in time and the fact that $\nabla^2_{\xi\xi}(0,\xi,\omega) = I_d$ holds for all $(\xi,\omega)$, it is shown that the Hessian in $(0,0)$ is invertible at all times.
\begin{proposition}\label{Proposition: Hessian in zero invertible}
	The Hessian $\nabla^2_{\xi\xi} \mathcal{V}(t,0,0)$ is positive definite for all $t \in [0,T]$.
\end{proposition}
\begin{proof}
	Since the set of symmetric positive definite matrices is open in the space of symmetric matrices with respect to any matrix norm $\Vert \cdot \Vert$, there exists some $\epsilon > 0$ such that any symmetric matrix $A$ satisfying $\Vert I_n - A \Vert < \epsilon$ is positive definite. Note that Proposition \ref{Proposition: Space derivatives continuous in time} shows that $\nabla^2_{\xi\xi} \mathcal{V}(\cdot,0,0)$ is continuous in $0$. Hence there exists some $t_0 \in (0,T)$ such that 
	\begin{equation*}
		\Vert I_n - \nabla^2_{\xi\xi} \mathcal{V}(t_0,0,0) \Vert
		= \Vert \nabla^2_{\xi\xi} \mathcal{V}(0,0,0) - \nabla^2_{\xi\xi} \mathcal{V}(t_0,0,0) \Vert < \epsilon,
	\end{equation*}
	implying that $\nabla^2_{\xi\xi} \mathcal{V}(t_0,0,0)$ is positive definite. With Proposition \ref{Proposition: DRE in zero without IV} it follows that $\nabla^2_{\xi\xi} \mathcal{V}(t,0,0)$ is the unique solution on $(0,T]$ of the differential Riccati equation
	\begin{equation*}
		\begin{aligned}
			\dot{\Pi}(t) &= \Pi(t) B(t) + B(t)^\top \Pi(t) - \Pi(t)FF^\top \Pi(t) + \alpha C^\top C,\\
			\Pi(t_0) &= \nabla^2_{\xi\xi} \mathcal{V}(t_0,0,0).
		\end{aligned}
	\end{equation*}
	Since $\nabla^2_{\xi\xi} \mathcal{V}(t_0,0,0)$ is positive definite, from \cite[Proposition 1.1]{DieEi94} we obtain that the Hessian $ \nabla^2_{\xi\xi} \mathcal{V}(t,0,0)$ is positive definite for all $t \in (0,T]$. The positive definiteness of the matrix $\nabla^2_{\xi\xi} \mathcal{V}(0,0,0) = I_n$ is clear and the assertion is shown.
\end{proof}
Finally it can be shown that for sufficiently small $(\xi,\omega)$ the Hessian is an invertible matrix.
\begin{theorem}\label{Theorem: Hessian invertible close to zero}
	There exists a constant $\delta_5 \in (0,\tfrac{1}{2} \delta_4 ]$ such that for any $t \in [0,T]$ and $(\xi,\omega) \in \mathcal{U}_{\delta_5}(0)$ the Hessian $\nabla^2_{\xi\xi} \mathcal{V}(t,\xi,\omega)$ is positive definite.
\end{theorem}
\begin{proof}
	The goal of this proof is to show that one can choose a time-independent $\delta_5$ such that for all $t \in [0,T]$ and $(\xi,\omega) \in \mathcal{U}_{\delta_5}(0)$ it holds
	\begin{equation*}
		\Vert \nabla^2_{\xi\xi}\mathcal{V}(t,\xi,\omega) - \nabla^2_{\xi\xi}\mathcal{V}(t,0,0) \Vert_{n,n}
		\leq \frac{1}{2} \Vert \nabla^2_{\xi\xi}\mathcal{V}(t,0,0)^{-1} \Vert_{n,n}^{-1}.
	\end{equation*}
	Then \cite[Proposition 7 in Section1.23]{Zei95AMS108} yields that $ \nabla^2_{\xi\xi}\mathcal{V}(t,\xi,\omega) $ is an invertible matrix.
	
	As a first step we will show continuity of the mapping $s \mapsto \nabla^2_{\xi\xi}\mathcal{V}(s,0,0)^{-1}$, which then implies existence of a minimum over $[0,T]$.
	First note that Proposition \ref{Proposition: Space derivatives continuous in time} ensures continuity of $s \mapsto \nabla^2_{\xi\xi} \mathcal{V}(s,0,0)$ in $[0,T]$. Since the mapping $A \mapsto \det(A)$ is a continuous mapping, we conclude that there exists $s^* \in [0,T]$ such that for all $t \in [0,T]$ it holds
	\begin{equation*}
		\det\left(\nabla_{\xi\xi}^2 \mathcal{V}(t,0,0)\right) \geq \min\limits_{s \in [0,T]} \det\left(\nabla_{\xi\xi}^2 \mathcal{V}(s,0,0)\right) = \det\left(\nabla_{\xi\xi}^2 \mathcal{V}(s^*,0,0)\right) > 0,
	\end{equation*}
	where the last inequality follows with the invertibility of $\nabla_{\xi\xi}^2 \mathcal{V}(s^*,0,0)$ shown in Proposition \ref{Proposition: Hessian in zero invertible}. Finally with Cramer's rule it follows that $s \mapsto \nabla^2_{\xi\xi} \mathcal{V}(s,0,0)^{-1}$ is continuous in $[0,T]$.
	
	For now set $\delta_5 = \tfrac{1}{2} \delta_4$. Let $t \in (0,T]$ be fixed and let $(\xi,\omega) \in \mathcal{U}_{\delta_5}(0)$. Let $\Vert \cdot \Vert_{n,n}$ denote the maximum norm on $\mathbb{R}^{n,n}$. Further let $i,j \in \{ 1,...,n \}$ be such that
	\begin{equation*}
		\max_{\substack{k,l \in\\ \{ 1,...,n \}}} \left\vert \frac{\partial^2}{\partial \xi_k \partial \xi_l} \mathcal{V}(t,\xi,\omega) - \frac{\partial^2}{\partial \xi_k \partial \xi_l}\mathcal{V}(t,0,0) \right\vert
		= \left\vert \frac{\partial^2}{\partial \xi_i \partial \xi_j} \mathcal{V}(t,\xi,\omega) - \frac{\partial^2}{\partial \xi_i \partial \xi_j}\mathcal{V}(t,0,0) \right\vert.
	\end{equation*}
	Taylor's theorem implies the existence of $\theta_\xi(t) \in [0,1]$ and $\theta_\omega(t) \in [0,1]$ such that
	\begin{equation*}
		\begin{aligned}
			&\Vert \nabla^2_{\xi\xi}\mathcal{V}(t,\xi,\omega) - \nabla^2_{\xi\xi}\mathcal{V}(t,0,0) \Vert_{n,n}
			= \left\vert \frac{\partial^2}{\partial \xi_i \partial \xi_j} \mathcal{V}(t,\xi,\omega) - \frac{\partial^2}{\partial \xi_i \partial \xi_j}\mathcal{V}(t,0,0) \right\vert\\
			&\leq \left\vert \frac{\partial^2}{\partial \xi_i \partial \xi_j} \mathcal{V}(t,\xi,\omega) - \frac{\partial^2}{\partial \xi_i \partial \xi_j}\mathcal{V}(t,0,\omega) \right\vert
			+ \left\vert \frac{\partial^2}{\partial \xi_i \partial \xi_j} \mathcal{V}(t,0,\omega) - \frac{\partial^2}{\partial \xi_i \partial \xi_j}\mathcal{V}(t,0,0) \right\vert\\
			&= \left\vert D_\xi \frac{\partial^2}{\partial \xi_i \partial \xi_j} \mathcal{V}(t, \theta_\xi(t) \xi, \omega) \xi \right\vert
			+ \left\vert D_\omega \frac{\partial^2}{\partial \xi_i \partial \xi_j} \mathcal{V}(t,0, \theta_\omega(t)\omega) \omega \right\vert\\
			&= \left\vert D^3_{\xi^3}\mathcal{V}(t,\theta_\xi(t)\xi,\omega)(\xi,e_i,e_j) \right\vert
			+ \left\vert D^3_{\omega \xi^2} \mathcal{V}(t,0,\theta_\omega(t) \omega)(\omega,e_i,e_j) \right\vert\\
			&\leq \check{M}_3 \Vert \xi \Vert + \check{M}_6 \Vert \omega \Vert_{L^2(0,t;\mathbb{R}^r)}
			\leq\left( \check{M}_3 + \check{M}_6 \right) \delta_5.
		\end{aligned}
	\end{equation*}
	Hence an appropriate decrease of $\delta_5$ ensures
	\begin{equation*}
		\Vert \nabla^2_{\xi\xi}\mathcal{V}(t,\xi,\omega) - \nabla^2_{\xi\xi}\mathcal{V}(t,0,0) \Vert_{n,n}
		\leq \frac{1}{2} \min\limits_{s \in [0,T]} \Vert \nabla^2_{\xi\xi}\mathcal{V}(s,0,0)^{-1} \Vert_{n,n}^{-1}
		\leq \frac{1}{2} \Vert \nabla^2_{\xi\xi}\mathcal{V}(t,0,0)^{-1} \Vert_{n,n}^{-1}.
	\end{equation*}
	Note that the minimum is a constant independent of $t$. Therefore $\delta_5$ can be chosen independent of $t$ as well. It is proven that $ \nabla^2_{\xi\xi}\mathcal{V}(t,\xi,\omega) $ is an invertible matrix. 
	
	To conclude the proof it remains to show that it is positive definite. Assuming the contrary implies that there exists a non-positive eigenvalue. Since the matrix is invertible it must be non-zero and therefore negative. Consider the function $f \colon [0,1] \rightarrow \mathbb{R}$, where $f(1)$ is that negative eigenvalue and $f(s)$ is the corresponding eigenvalue of $\nabla_{\xi\xi}^2 \mathcal{V}(t,s\xi,s\omega)$. Since eigenvalues depend continuously on the matrix entries, $f$ is a continuous function. As shown earlier $\nabla_{\xi\xi}^2 \mathcal{V}(t,0,0)$ is a symmetric positive definite matrix implying $f(0) > 0$. By assumption it holds $f(1) < 0$. With the continuity of $f$ it follows the existence of $s \in (0,1)$ such that $f(s) = 0$. Then $\nabla_{\xi\xi}^2\mathcal{V}(t,s\xi,s\omega)$ has a zero eigenvalue and is therefore not invertible. This leads to a contradiction because $ s(\xi,\omega) \in \mathcal{U}_{\delta_5}(0)$. 
\end{proof}
\begin{remark}~
	\begin{itemize}
		\item[(i)] Note that it is essential to enforce $\delta_5 < \delta_4$. Therefore a time-dependence of $\delta_4$ would have carried over to $\delta_5$.
		\item[(ii)] To obtain the positive definiteness of $\nabla_{\xi\xi}^2 \mathcal{V}(t,0,0)$ in Proposition \ref{Proposition: Hessian in zero invertible} we heavily depend on the fact that $\omega = 0 \in L^2(0,T;\mathbb{R}^r)$ is continuous and essentially bounded. Only due to this regularity Theorem \ref{Theorem: HJB holds} ensures that the HJB equation holds and eventually leads to Proposition \ref{Proposition: DRE in zero without IV}. However, in Theorem \ref{Theorem: Hessian invertible close to zero} we only assume, that $\omega \in L^2(0,T;\mathbb{R}^r)$ has a sufficiently small norm. The assumption of continuity is no longer required. 
	\end{itemize}
\end{remark}
%---------------------------------------------------------------
%---------------------------------------------------------------
%---------------------------------------------------------------
\section{The Mortensen observer}\label{sec: mortensen}
%---------------------------------------------------------------
%---------------------------------------------------------------
%---------------------------------------------------------------
This section contains the main result of this work. We will show that for sufficiently small $\omega$ and fixed arbitrary $t$ the value function admits a unique minimizer in $\xi$. The resulting trajectory will be characterized as a solution of the observer equation.
\begin{lemma}\label{lem: VFminimizer}
	There exist constants $\delta_6 \in (0,\delta_5]$ and $\hat{c}_1 > 0 $ such that for every $t \in [0,T]$ and $\omega \in L^2(0,t;\mathbb{R}^r)$ with $\Vert \omega \Vert_{L^2(0,t;\mathbb{R}^r)} < \delta_6$ the following holds: The mapping $\xi \mapsto \mathcal{V}(t,\cdot,\omega)$ admits at most one global minimizer and it holds
	\begin{equation*}
		\xi \in \mathbb{R}^n \text{ is a minimizer of } \mathcal{V}(t,\cdot,\omega)
		\iff \nabla_\xi \mathcal{V}(t,\xi,\omega) = 0 \text{ and } \Vert \xi \Vert < \hat{c}_1 \delta_6.
	\end{equation*}
\end{lemma}
\begin{proof}
	For now set $\delta_6 = \delta_5$ and let $t \in (0,T]$ be arbitrary. First note that 
	\begin{equation*}
		\inf_\xi \mathcal{V}(t,\xi,\omega) 
		\leq \mathcal{V}(t,0,\omega)
		= \min_{\substack{x,v \text{ st. } \\ e(x,v;t,0)=0}} J(x,v;t,\omega)
		\leq J(\Tilde{x},0;t,\omega)
		\leq \tfrac{\alpha}{2} \Vert \omega \Vert_{L^2(0,t;\mathbb{R}^r)}^2
		< \tfrac{\alpha}{2} \delta_6^2.
	\end{equation*}
	Assume $\xi^* \in \mathbb{R}^n$ to be a minimizer. It follows 
	\begin{equation*}
		\inf_{\substack{x,v \text{ st. } \\ e(x,v;t,\xi^*)=0}} J(x,v;t,\omega)
		= \mathcal{V}(t,\xi^*,\omega) < \tfrac{\alpha}{2}\delta_6^2.
	\end{equation*}
	Hence there exist $x \in V_t$ and $v \in L^2(0,t;\mathbb{R}^m)$ such that 
	$e(x,v;t,\xi^*) = 0$ and $J(x,v;t,\omega) < \tfrac{\alpha}{2} \delta_6^2$. Then $\tfrac{1}{2} \Vert x(0) - \Tilde{x}(0) \Vert^2 < \tfrac{\alpha}{2} \delta_6^2$ and $\tfrac{1}{2} \Vert v \Vert_{L^2(0,t;\mathbb{R}^m)}^2 < \tfrac{\alpha}{2} \delta_6^2$. With the argument used in the proof of Proposition \ref{Proposition:ErrorSolvability} and a possible decrease of $\delta_6$ it can be shown that there exists $c_1 > 0$ independent of $t$, $x$, $v$, $\omega$ and $\xi^*$ such that
	\begin{equation*}
		\Vert x - \Tilde{x} \Vert_{L^\infty(0,t;\mathbb{R}^n)}
		\leq c_1 \left( \Vert x(0) - \Tilde{x}(0) \Vert + \Vert v \Vert_{L^2(0,t;\mathbb{R}^m)} \right).
	\end{equation*}
	It follows
	\begin{equation*}
		\begin{aligned}
			\tfrac{1}{2} \Vert \xi^* \Vert^2
			&\leq \tfrac{1}{2} \Vert x(t) - \Tilde{x}(t) \Vert^2
			\leq \tfrac{1}{2} \Vert x - \Tilde{x} \Vert_{L^\infty(0,t;\mathbb{R}^n)}^2
			\leq c_1^2 \left( \Vert x(0) - \Tilde{x}(0) \Vert^2 + \Vert v \Vert_{L^2(0,t;\mathbb{R}^m)}^2 \right)\\
			&\leq 2c_1^2 J(x,v;t,\omega)
			< c_1^2 \alpha \delta_6^2.
		\end{aligned}
	\end{equation*}    
	Therefore the search for a minimizer can be restricted to the open convex set \begin{equation*}
		\mathcal{K} \coloneqq \left\{ \xi \in \mathbb{R}^n \colon \Vert \xi \Vert < \sqrt{2\alpha}c_1\delta_6 \right\}.
	\end{equation*}
	A possible decrease of $\delta_6$ ensures $\mathcal{K} \subset \mathcal{U}_{\delta_5}(0)$. According to Theorem \ref{Theorem: Hessian invertible close to zero} the Hessian $\nabla_{\xi\xi}^2\mathcal{V}(t,\xi,\omega)$ is positive definite for all $\xi \in \mathcal{K}$, implying that $\mathcal{V}(t,\cdot,\omega)$ is strictly convex on $\mathcal{K}$. The assertion is shown for $t \in (0,T]$. Noting that $\mathcal{V}(0,\xi,\omega) =\tfrac{1}{2} \Vert \xi \Vert^2$ concludes the proof.
\end{proof}
For any fixed $t\in(0,T]$ an application of the implicit function theorem to $\nabla_\xi\mathcal{V}(t,\cdot,\cdot)$ shows existence of a unique minimizer.
\begin{proposition}\label{prop: WDargmin}
	There exists a constant $\delta_7 \in (0,\delta_6]$ such that for every $t \in [0,T]$ and $\omega \in L^2(0,t;\mathbb{R}^m)$ with $\Vert \omega \Vert_{L^2(0,t;\mathbb{R}^m)} < \delta_7 $ the mapping $\xi \mapsto \mathcal{V}(t,\xi,\omega)$ admits exactly one minimizer $\xi^* \in \mathbb{R}^n$. Further $\xi^*$ depends continuously on $\omega$ and there exists a constant $\hat{c}_2 > 0$ such that it holds $\Vert \xi^* \Vert \leq \hat{c}_2 \delta_7$.
\end{proposition}
\begin{proof}
	Due to Lemma \ref{lem: VFminimizer} we only have to show existence of $\delta_7 \in (0,\delta_6]$ such that for any $t \in (0,T]$ and $\omega \in L^2(0,t;\mathbb{R}^r)$ satisfying $\Vert \omega \Vert_{L^2(0,t;\mathbb{R}^r)} < \delta_7$ there exists $\xi^* \in \mathcal{U}_{\hat{c}\delta_6}(0)$ with $\nabla_\xi \mathcal{V}(t,\xi^*,\omega) = 0$. For a fixed $t \in (0,T]$ we will apply \cite[Theorem]{Hol70} to the mapping $(\xi,\omega) \mapsto P(\xi,\omega) \coloneqq \nabla_\xi \mathcal{V}(t,\xi,\omega)$, which is a continuous map on $\Omega = \mathcal{U}_{\delta_5}(0) \subset \mathbb{R}^n \times L^2(0,t;\mathbb{R}^r)$.
	
	First note that $D_{\xi} P = \nabla_{\xi\xi}^2\mathcal{V}(t,\cdot,\cdot)$ exists and is continuous in $\Omega$. According to Theorem \ref{Theorem: Hessian invertible close to zero} the derivative $D_\xi P(0,0) = \nabla_{\xi\xi}^2\mathcal{V}(t,0,0)$ is invertible and with Lemma \ref{lem: EstHesInv} it holds $\Vert \nabla_{\xi\xi}^2\mathcal{V}(t,0,0)^{-1} \Vert_2 < M_\mathcal{V}$. Since all matrix norms are equivalent, there exists a constant $c_1>0$ such that $\Vert A \Vert_2 \leq c_1\Vert A \Vert_{n,n}$ for all $A \in \mathbb{R}^{n,n}$. Hence the arguments made in the proof of Theorem \ref{Theorem: Hessian invertible close to zero} yield that for all $(\xi,\omega) \in \Omega$ it holds
	\begin{equation*}
		\Vert \nabla_{\xi\xi}^2\mathcal{V}(t,\xi,\omega) - \nabla_{\xi\xi}^2\mathcal{V}(t,0,0) \Vert_2
		\leq c_1 \left(\check{M}_3 \Vert \xi \Vert + \check{M}_6 \Vert \omega \Vert_{L^2(0,t;\mathbb{R}^r)} \right) \eqqcolon g_1(\Vert \xi \Vert, \Vert \omega \Vert_{L^2(0,t;\mathbb{R}^r)}).
	\end{equation*}
	After noting that it holds $\nabla_\xi \mathcal{V}(t,0,0) = 0$ a similar argument shows that for all $\omega \in \mathcal{U}_{\delta_5}(0)$ one has
	\begin{equation*}
		\Vert \nabla_\xi \mathcal{V}(t,0,\omega) \Vert 
		\leq c_2 \check{M}_5 \Vert \omega \Vert_{L^2(0,t;\mathbb{R}^r)}
		\eqqcolon g_2(\Vert \omega \Vert_{L^2(0,t;\mathbb{R}^r)}),
	\end{equation*}
	where $c_2 > 0$ is such that $\Vert \xi \Vert \leq c_2 \max\{ \vert \xi_i \vert \colon i=1,...,n \}$ for all $\xi \in \mathbb{R}^n$.
	
	For the verification of the last assumption made in \cite{Hol70} we set $\alpha = \tfrac{1}{2}$, $\delta_7 = \delta_6$ and $\epsilon(\delta_7) = 2c_2M_\mathcal{V}\check{M}_5\delta_7$. A possible decrease of $\delta_7$ ensures $\epsilon(\delta_7) < \delta_5$ and 
	\begin{equation*}
		\delta_7 M_\mathcal{V}c_1 \left( \check{M}_6 + 2c_2 \check{M}_3 \check{M}_5 M_\mathcal{V} \right) \leq \tfrac{1}{2}.
	\end{equation*}
	Now \cite[Theorem]{Hol70} delivers the existence of a continuous operator $F $ that maps $\mathcal{U}_{\delta_7}(0) \subset L^2(0,t;\mathbb{R}^r)$ into $\overline{\mathcal{U}_{\epsilon(\delta_7)}(0)}$ with the property $\nabla_\xi \mathcal{V}(t,F(\omega),\omega) = 0$. Another decrease of $\delta_7$ ensures $\epsilon(\delta_7) < \hat{c}_1 \delta_6$. Now for any $\omega \in L^2(0,t;\mathbb{R}^r)$ with $\Vert \omega \Vert_{L^2(0,t;\mathbb{R}^r)} < \delta_7$ there exists $\xi^* = F(\omega)$ satisfying $\Vert \xi^* \Vert < \hat{c}\delta_6 $ and $\nabla_\xi \mathcal{V}(t,\xi^*,\omega) = 0$. Lemma \ref{lem: VFminimizer} yields the assertion for $t \in (0,T]$. The result obviously holds for $t = 0$.
\end{proof}
With this result at hand we can define $\widehat{x}(t)$ as the trajectory given pointwise in time as the minimizer of the value function with respect to $\xi$.
\begin{corollary}\label{cor: xHatWD}
	Assume $\omega \in L^2(0,T;\mathbb{R}^r)$ satisfies $\Vert \omega \Vert_{L^2(0,T;\mathbb{R}^r)} < \delta_7$.  Then for all $t \in [0,T]$ the expression $\argmin\limits_{\xi} \mathcal{V}(t,\xi,\omega) \in \mathbb{R}^n$ is well defined and we define the function 
	\begin{equation*}
		\widehat{x}_\omega \colon [0,T] \rightarrow \mathbb{R}^n,~~~~~
		\widehat{x}_\omega(t) = \argmin\limits_{\xi} \mathcal{V}(t,\xi,\omega).
	\end{equation*}
	For all $t \in [0,T]$ it satisfies $\Vert \widehat{x}_\omega(t) \Vert \leq \hat{c}_2 \delta_7 < \hat{c}_1 \delta_6 < \delta_5$.
\end{corollary}

In the following we will characterize this trajectory as a solution of the observer equation. We first do this for continuous data $\omega$ before using a density argument to extend the result to $\omega \in L^2(0,T;\mathbb{R}^r)$. As a first step it is shown that for continuous $\omega$ the trajectory $\widehat{x}_\omega$ is differentiable in $(0,T)$.

\begin{lemma}\label{lem: xHatC1}
	Assume $\omega \in L^2(0,T;\mathbb{R}^r)$ satisfies $\Vert \omega \Vert_{L^2(0,T;\mathbb{R}^r)} < \delta_7$ and is continuous in $(0,T)$. Then $\widehat{x}_\omega$ is differentiable in all $t\in (0,T)$. For any $t \in (0,T)$ it holds
	\begin{equation*}
		\dot{\widehat{x}}_\omega(t) 
		= h(t,\widehat{x}_\omega(t)) + \alpha \nabla_{\xi\xi}^2 \mathcal{V}(t,\widehat{x}_\omega(t),\omega)^{-1} C^\top (\omega(t) - C \widehat{x}_\omega(t)), 
	\end{equation*}
	where again $h(t,\xi) = A\xi + G(\tilde{x}(t) \otimes I_n)\xi + G(I_n \otimes \tilde{x}(t))\xi + G(\xi \otimes \xi)$.
\end{lemma}
\begin{proof}
	With Lemma \ref{Lemma: Switching Derivatives} (i) it holds that the mappings $(t,\xi) \mapsto \nabla_\xi \mathcal{V}(t,\xi,\omega)$ and $(t,\xi) \mapsto \nabla_{\xi\xi}^2 \mathcal{V}(t,\xi,\omega)$ are continuous in $(0,T) \times \mathcal{U}_{\delta_5}(0)$. Note that $\Vert \widehat{x}_\omega(t) \Vert < \delta_5$. The arguments used in the proof of Lemma \ref{Lemma: Switching Derivatives} further show that $\nabla_\xi \mathcal{V}(\cdot,\cdot,\omega)$ is differentiable in $t$ and that $\partial_t \nabla_\xi \mathcal{V}(t,\xi,\omega) = \nabla_\xi \partial_t  \mathcal{V}(t,\xi,\omega)$. With the HJB equation we get that $(t,\xi) \mapsto \partial_t \nabla_\xi \mathcal{V}(t,\xi,\omega)$ is continuous in $(0,T) \times \mathcal{U}_{\delta_5}(0)$. Hence $\nabla_\xi \mathcal{V}(\cdot,\cdot,\omega) \in C^1((0,T) \times \mathcal{U}_{\delta_5}(0);\mathbb{R}^n)$. Let $t_0 \in (0,T)$ be arbitrary. Applying the implicit function theorem \cite[Theorem 4.E]{Zei95AMS109} in the point $(t_0,\widehat{x}_{\omega}(t_0))$ to the mapping $P(t,\xi) = \nabla_\xi \mathcal{V}(t,\xi,\omega)$ shows that $\widehat{x}_\omega$ is differentiable in $t_0$.
	
	To obtain the formula for the derivative note that for all $t \in (0,T)$ it holds 
	\begin{equation}\label{eq:gradzero}
		\nabla_\xi \mathcal{V}(t,\widehat{x}_\omega(t),\omega) = 0. 
	\end{equation}
	Taking the derivative with respect to $t$ and applying the chain rule yields
	\begin{equation*}
		\nabla_{\xi\xi}^2\mathcal{V}(t,\widehat{x}_\omega(t),\omega)~ \dot{\widehat{x}}_\omega(t)
		= - \partial_t \nabla_\xi \mathcal{V}(t,\widehat{x}_\omega(t),\omega).
	\end{equation*}
	As mentioned above we can change the order of derivation on the right hand side. With the HJB and \eqref{eq:gradzero} it follows
	\begin{equation*}
		\nabla_{\xi\xi}^2\mathcal{V}(t,\widehat{x}_\omega(t),\omega)~ \dot{\widehat{x}}_\omega(t)
		= \nabla_{\xi\xi}^2\mathcal{V}(t,\widehat{x}_\omega(t),\omega)~ h(t,\widehat{x}_\omega(t))
		+ \alpha C^\top (\omega(t) - C \widehat{x}_\omega(t)).
	\end{equation*}
	The bounds on $\Vert \omega \Vert_{L^2(0,T;\mathbb{R}^r)}$ and $\Vert \widehat{x}_\omega(t) \Vert$ ensure the exitence of the inverse of the Hessian and the assertion is shown.
\end{proof}
With this lemma we can show that for continuous $\omega$ the trajectory $\widehat{x}_\omega$ is a weak solution of \eqref{eq:ObsShif}.
\begin{proposition}\label{prop: obsEqContW}
	Assume $\omega \in L^2(0,T;\mathbb{R}^r)$ satisfies $\Vert \omega \Vert_{L^2(0,T;\mathbb{R}^r)} < \delta_7$ and is continuous in $[0,T]$. Then $\widehat{x}_\omega$ lies in $H^1(0,T;\mathbb{R}^n)$ and is a weak solution to the observer equation \eqref{eq:ObsShif}.
\end{proposition}
\begin{proof}
	Due to Lemma \ref{lem: xHatC1} we already know that $\widehat{x}_\omega$ is differentiable almost everywhere in $[0,T]$ with a derivative that lies in $L^2(0,T;\mathbb{R}^n)$. Furthermore it holds $\widehat{x}_\omega(0) = \argmin\limits_\xi \tfrac{1}{2} \Vert \xi \Vert^2 = 0$. It remains to show that $\widehat{x}_\omega$ is absolutely continuous in $[0,T]$. To that end we first show that it is continuous in $t = 0$.
	
	Let $(t_k) \subset (0,T)$ be an arbitrary sequence with the property $t_k \to 0$ for $k \to \infty$. We will show that $\widehat{x}_\omega(t_k) \to 0 = \widehat{x}_\omega(0)$ for $k \to \infty$. First note that for all $k$ it holds $\Vert \widehat{x}_\omega(t_k) \Vert < \hat{c}_2 \delta_7$ which implies that $(\widehat{x}_\omega(t_k))$ is a bounded sequence in $\mathbb{R}^n$. Hence any arbitrary subsequence admits a converging subsequence. It will be denoted by $(\widehat{x}_\omega(t_k))$ with the limit $\widehat{x} \in \mathbb{R}^n$. With Lemma \ref{Lemma: Switching Derivatives} (i) it holds that $(t,\xi) \mapsto \nabla_\xi\mathcal{V}(t,\xi,\omega)$ is continuous in $[0,T] \times \mathcal{U}_{\delta_5}(0)$. Then it follows
	\begin{equation*}
		\widehat{x} 
		= \nabla_\xi \mathcal{V}(0,\widehat{x},\omega)
		= \lim\limits_{k \to \infty} \nabla_\xi \mathcal{V}(t_k,\widehat{x}_\omega(t_k),\omega) = 0,
	\end{equation*}
	where the first equality is due the fact that in $t = 0$ it holds $\mathcal{V}(0,\xi,\omega) = \tfrac{1}{2} \Vert \xi \Vert^2$ for all $\xi \in \mathbb{R}^n$. This implies that every subsequence of $\widehat{x}_\omega(t_k)$ admits a subsequence that converges to $0$. Hence the sequence itself converges to zero. Therefore $\widehat{x}_\omega$ is continuous in $t = 0$. It remains to show that it is absolutely continuous in $[0,T]$.
	
	To that end let $\epsilon > 0$ be arbitrary. Then there exists $\delta > 0$ such that it holds
	\begin{equation*}
		\vert a \vert < \delta \implies \Vert \widehat{x}_\omega(a) \Vert < \tfrac{1}{2} \epsilon. 
	\end{equation*}
	Further note that $\dot{\widehat{x}}_\omega$ is bounded on $(0,T)$. This is due to Lemma \ref{lem: xHatC1}, the bound $\Vert \widehat{x}_\omega(t) \Vert \leq \hat{c}_2\delta_7$ and the assumed continuity of $\omega$. Let $(a_k,b_k)$ be a finite sequence of pairwise disjoint subintervals from $[0,T]$ with $a_k < b_k$ for $k = 1,...,N$. Without loss of generality assume that $a_1 = 0$. With Taylor's Theorem it holds
	\begin{equation*}
	\begin{aligned}
		\sum_{k=1}^N \Vert \widehat{x}_\omega(b_k) - \widehat{x}_\omega(a_k) \Vert
		&= \Vert \widehat{x}_\omega(b_1) \Vert 
		+ \sum_{k=2}^N \Vert \widehat{x}_\omega(b_k) - \widehat{x}_\omega(a_k) \Vert
		\leq \Vert \widehat{x}_\omega(b_k) \Vert \\
		&+ \max_{s \in (0,T)} \Vert \dot{\widehat{x}}_\omega(s) \Vert \sum_{k=2}^N b_k - a_k.
	\end{aligned}
	\end{equation*}
	After an appropriate decrease of $\delta$ it holds $\delta \max_{s \in (0,T)} \Vert \dot{\widehat{x}}_\omega(s) \Vert < \tfrac{1}{2} \epsilon$. Then for any such finite sequence of subintervals it holds
	\begin{equation*}
		\sum_{k=1}^N b_k - a_k < \delta
		\implies \sum_{k=1}^N \Vert \widehat{x}_\omega(b_k) - \widehat{x}_\omega(a_k) \Vert < \epsilon,
	\end{equation*}
	which shows that $\widehat{x}_\omega$ is absolutely continuous in $[0,T]$. The almost everywhere existing classical derivative given in Lemma \ref{lem: xHatC1} is an element of $L^2(0,t;\mathbb{R}^n)$. For one-dimensional domains this implies that $\widehat{x}_\omega$ admits a weak derivative which almost everywhere agrees with the classical one and the assertion follows.
\end{proof}
Finally a density argument is used to show that the assertion of Proposition \ref{prop: obsEqContW} also holds for $\omega \in L^2(0,T;\mathbb{R}^r)$.
\begin{theorem}\label{thm: MainThm}
	Assume $\omega \in L^2(0,T;\mathbb{R}^r)$ satisfies $\Vert \omega \Vert_{L^2(0,T;\mathbb{R}^r)} < \delta_7$. Then $\widehat{x}_\omega$ lies in $H^1(0,T;\mathbb{R}^n)$ and is a weak solution to the observer equation \eqref{eq:ObsShif}.
\end{theorem}
\begin{remark}\label{rem: MainThm}
	Before we turn to the main result of this work, some remarks are in order.
	The assumptions of Theorem \ref{thm: MainThm} also underline the importance of avoiding  time-dependence of the constants throughout this work.
	If the constant $\delta_7$ was time-dependent with the property $\delta_7(t) \to 0, t \to 0$ the result would be considerably weaker. One would have to assume that for all $t \in (0,T]$ it holds
	\begin{equation*}
		\int_0^t \Vert \omega(s) \Vert^2 \,\mathrm{d}s < \delta_7(t)^2.
	\end{equation*}
	This means in particular that any continuous $\omega$ would need to satisfy $\omega(0) = 0$. Translating this to the original coordinates and setting used in \eqref{eq:state_intro} and \eqref{eq:output_intro} this would imply $ C \eta = - \mu(0) $, meaning that the modeled observation of the disturbance in the initial state and the observation error in time zero line up. It would also significantly increase the technical difficulties in the subsequent proof.
\end{remark}
\begin{proof}
	Let $\omega_k$ be a sequence from $C([0,T];\mathbb{R}^r)$ such that $\Vert \omega - \omega_k \Vert_{L^2(0,T;\mathbb{R}^r)} \to 0$ for $k \to 0$ and $\Vert \omega_k \Vert_{L^2(0,T;\mathbb{R}^r)} < \delta_7 $ for all $k$. Consider the sequence $(\widehat{x}_{\omega_k}) \subset H^1(0,T;\mathbb{R}^n)$. According to Proposition \ref{prop: WDargmin} for any fixed $t \in [0,T]$ the mapping $\omega \mapsto \widehat{x}_\omega(t)$ is continuous implying that $\widehat{x}_{\omega_k}$ converges to $\widehat{x}_\omega$ pointwise. For all $k \in \mathbb{N}$ and $t \in [0,T]$ it holds $ \Vert \widehat{x}_{\omega_k}(t) \Vert \leq \hat{c}_2 \delta_7$, hence the dominated convergence theorem implies that for all $\varphi \in C_0^\infty(0,T;\mathbb{R}^n)$ it holds 
	\begin{equation*}
		\int_0^T \left\langle \widehat{x}_\omega(t),\dot{\varphi}(t) \right\rangle\,\mathrm{d}t
		= \lim_{k \to \infty} \int_0^T \left\langle \widehat{x}_{\omega_k}(t),\dot{\varphi}(t) \right\rangle \,\mathrm{d}s
		= - \lim_{k \to \infty} \int_0^T \left\langle \dot{\widehat{x}}_{\omega_k}(t),\varphi(t) \right\rangle \,\mathrm{d}t,
	\end{equation*}
	where $\dot{\widehat{x}}_{\omega_k}$ is given by the right hand side of the observer equation. It remains to show that
	\begin{equation*}
		\begin{aligned}
			&\int_0^T \left\langle
			h(t,\widehat{x}_\omega(t)) + \alpha \nabla_{\xi\xi}^2\mathcal{V}(t,\widehat{x}_\omega(t),\omega)^{-1} C^\top (\omega(t) - C\widehat{x}_\omega(t))
			,\varphi(t) \right\rangle \,\mathrm{d}t\\
			&-\int_0^T \left\langle
			h(t,\widehat{x}_{\omega_k}(t)) + \alpha \nabla_{\xi\xi}^2\mathcal{V}(t,\widehat{x}_{\omega_k}(t),\omega_k)^{-1} C^\top (\omega_k(t) - C\widehat{x}_{\omega_k}(t))
			,\varphi(t) \right\rangle \,\mathrm{d}t
		\end{aligned}
	\end{equation*}
	converges to zero for $k \to \infty$. First consider
	\begin{equation*}
		\left\vert \int_0^T \left\langle h(t,\widehat{x}_{\omega}(t)) - h(t,\widehat{x}_{\omega_k}(t)), \varphi(t) \right\rangle \,\mathrm{d}t \right\vert
		\leq \Vert \varphi \Vert_{L^\infty(0,T;\mathbb{R}^n)} \int_0^T \Vert h(t,\widehat{x}_{\omega}(t)) - h(t,\widehat{x}_{\omega_k}(t)) \Vert \,\mathrm{d}t.
	\end{equation*}
	Note that for any $\xi \in \mathbb{R}^n$ and $t \in [0,T]$ it holds $h(t,\xi) \leq \left(\Vert A \Vert_2 + 2\Vert G \Vert_2 \Vert \Tilde{x} \Vert_{L^\infty(0,T;\mathbb{R}^n)}\right) \Vert \xi \Vert + \Vert G \Vert_2 \Vert \xi \Vert^2 $. Further $h$ is continuous in its second argument, hence the pointwise convergence of $\widehat{x}_{\omega_k}$ and dominated convergence imply convergence to zero.  Next consider
	\begin{equation*}
		\begin{aligned}
			&\left\vert \int_0^T \left\langle 
			\nabla_{\xi\xi}^2\mathcal{V}(t,\widehat{x}_{\omega}(t),\omega)^{-1}C^\top \omega(t)
			- \nabla_{\xi\xi}^2\mathcal{V}(t,\widehat{x}_{\omega_k}(t),\omega_k)^{-1}C^\top \omega_k(t)
			,\varphi(t) \right\rangle \,\mathrm{d}t \right\vert\\
			&\leq \Vert \varphi \Vert_{L^\infty(0,T;\mathbb{R}^n)}
			\int_0^T \Vert \nabla_{\xi\xi}^2\mathcal{V}(t,\widehat{x}_{\omega}(t),\omega)^{-1}C^\top \omega(t)
			- \nabla_{\xi\xi}^2\mathcal{V}(t,\widehat{x}_{\omega_k}(t),\omega_k)^{-1}C^\top \omega_k(t) \Vert \,\mathrm{d}t.
		\end{aligned}
	\end{equation*}
	The integral can be estimated from above by
	\begin{equation*}
		\begin{aligned}
			&\Vert C \Vert_2 \int_0^T \Vert \nabla_{\xi\xi}^2\mathcal{V}(t,\widehat{x}_\omega(t),\omega)^{-1} 
			- \nabla_{\xi\xi}^2\mathcal{V}(t,\widehat{x}_{\omega_k}(t),\omega_k)^{-1} \Vert_2 \Vert \omega(t) \Vert \,\mathrm{d}t\\
			&+ \Vert C \Vert_2 \int_0^T \Vert \nabla_{\xi\xi}^2\mathcal{V}(t,\widehat{x}_{\omega_k}(t),\omega_k)^{-1} \Vert_2 \Vert \omega(t) - \omega_k(t) \Vert \,\mathrm{d}t.
		\end{aligned}
	\end{equation*}
	With the uniform bound of $\Vert \nabla_{\xi\xi}^2\mathcal{V}(t,\widehat{x}_{\omega_k}(t),\omega_k)^{-1} \Vert_2$ from Lemma \ref{lem: EstHesInv} and the convergence of $\omega_k$ in $L^2$ the second summand converges to zero. Lemma \ref{lem: EstHesInv} further implies that for all $k \in \mathbb{N}$ and $t \in [0,T]$ it holds
	\begin{equation*}
		\Vert \nabla_{\xi\xi}^2\mathcal{V}(t,\widehat{x}_\omega(t),\omega)^{-1} 
		- \nabla_{\xi\xi}^2\mathcal{V}(t,\widehat{x}_{\omega_k}(t),\omega_k)^{-1} \Vert_2
		\leq 2M_\mathcal{V}.
	\end{equation*}
	Since $\mathcal{V}$ is $C^\infty$ in $\xi$ and $\omega$, the dominated convergence theorem yields that the first summand also converges to zero. The convergence of the term involving $-C \widehat{x}_\omega$ instead of $\omega$ can be shown analogously. Hence it is shown that the weak derivative of $\widehat{x}_\omega$ is given by 
	\begin{equation*}
		h(t,\widehat{x}_\omega(t)) + \alpha \nabla_{\xi\xi}^2\mathcal{V}(t,\widehat{x}_\omega(t),\omega)C^\top(\omega(t)-C\widehat{x}_\omega(t)) \in L^2(0,T;\mathbb{R}^n).
	\end{equation*}
	By definition it holds $\widehat{x}_\omega(0) = 0$, therefore $\widehat{x}_\omega \in H^1(0,T;\mathbb{R}^n)$ is a weak solution of \eqref{eq:ObsShif}.    
\end{proof} 

\section{Conclusion}
In this work we were able to show that the formal derivation of the observer equation associated with the Mortensen observer can be done in a rigorous manner, assuming that the initial model is close enough to the actual system from which the data is measured. These results motivate and justify two approaches to approximate the non-linear observer numerically. On the one hand the observer trajectory can be obtained by a minimization of the value function in every time point, on the other hand it is given as a solution of the observer equation. 
%%%
%END%
%%%

\section*{Acknowledgement}

We thank K.~Kunisch (KFU Graz) for many helpful comments and several discussions on earlier versions of this manuscript.
We thank L.~Pfeiffer (INRIA Paris-Saclay) for many helpful suggestions and pointers on the subject of sensitivity analysis.
We gratefully acknowledge funding and support from the Deutsche Forschungsgemeinschaft via the project 504768428.

\appendix
\section{Linear quadratic control problem}\label{Appendix A}
This section of the Appendix is concerned with a general linear quadratic control problem. Its solution is characterized as the solution of a system of equations. Further the optimizing triple of state, control and adjoint is estimated by the data. The control problem is given as
\begin{equation}\label{LQCP}
	\min\limits_{(x,v) \in V_t \times L^2(0,t;\mathbb{R}^m)} J(x,v), 
	\text{ subject to }
	e(x,v) = 0,
\end{equation}
where the cost functional $J \colon V_t \times L^2(0,t;\mathbb{R}^m) \rightarrow \mathbb{R}$ and the constraint $e \colon V_t \times L^2(0,t;\mathbb{R}^m) \rightarrow L^2(0,t;\mathbb{R}^n) \times \mathbb{R}^n$ are defined as
\begin{equation}
	\begin{aligned}
		J(x,v) 
		&= \frac{1}{2} \Vert x(0) - b \Vert^2 
		+ \frac{1}{2} \Vert v \Vert_{L^2(0,t;\mathbb{R}^m)}^2 
		+ \frac{\alpha}{2} \Vert Cx - \mu \Vert_{L^2(0,t;\mathbb{R}^r)}^2 \\
		&+ \int_0^t \left\langle G(x \otimes x) , \rho \right\rangle + \left\langle l_1 , x \right\rangle + \left\langle l_2 , v \right\rangle \,\mathrm{d}s,\\
		e(x,v)
		&= \left( \dot{x} -Ax -G(\breve{x}_t \otimes I_n) x - G(I_n \otimes \breve{x}_t) x - Fv -f, x(t) - a \right).
	\end{aligned}
\end{equation}
Here $t \in (0,T]$ and the data consisting of $a,b \in \mathbb{R}^n$, $\mu \in L^2(0,t;\mathbb{R}^r)$, $f \in L^2(0,t;\mathbb{R}^n)$ and $l_1 \in L^2(0,t;\mathbb{R}^n)$, $l_2 \in L^2(0,t;\mathbb{R}^m)$ is given.
\begin{proposition}\label{Proposition: Appendix A}
	Assume that $\breve{x}_s \in V_s$ is well-defined for any $s \in (0,T]$ and that $\Vert \breve{x}_s \Vert_{L^\infty(0,s;\mathbb{R}^n)} \leq \kappa$ for some time-independent constant $\kappa> 0$ and all $s \in (0,T]$ . Further assume that $\rho \in V_t$ satisfies $\Vert \rho \Vert_{L^\infty(0,t;\mathbb{R}^n)} \leq \beta \coloneqq (4 \bar{c}(\kappa)^2 \Vert G \Vert_2)^{-1}$. Then the system
	\begin{equation}\label{Appendix KKT system}
		\begin{aligned}
			\dot{x} &= Ax + G(\breve{x}_t \otimes I_n) x + G(I_n \otimes \breve{x}_t)x + Fv + f,\\
			x(t) &= a,\\
			\dot{q} &= -A^\top q - (\breve{x}_t^\top \otimes I_n) G^\top q - (I_n \otimes \breve{x}_t^\top)G^\top q
			- \alpha C^\top (Cx - \mu)\\
			&- (x^\top \otimes I_n)G^\top \rho - (I_n \otimes x^\top)G^\top \rho - l_1,\\
			q(0) &= b - x(0),\\
			v + F^\top q + l_2 &= 0,
		\end{aligned}
	\end{equation}
	admits exactly one solution $(\bar{x},\bar{v},\bar{q})$. It is given by the minimizing control of \eqref{LQCP} and its corresponding state and adjoint state. Further there exists a constant $\overline{M} > 0$ independent of $t$, $a$, $b$, $\mu$, $f$, $l_1$, $l_2$, $\rho$ und $\breve{x}_t$ such that
	\begin{equation*}
		\begin{aligned}
			&\max \left(
			\Vert \bar{v} \Vert_{L^2(0,t;\mathbb{R}^m)},
			\Vert \bar{x} \Vert_{V_t},
			\Vert \bar{q} \Vert_{V_t}
			\right)\\ 
			&\leq \overline{M} \left( \Vert a \Vert + \Vert b \Vert + \Vert \mu \Vert_{L^2(0,t;\mathbb{R}^m)} + \Vert f \Vert_{L^2(0,t;\mathbb{R}^n)} + \Vert l_1 \Vert_{L^2(0,t;\mathbb{R}^n)} + \Vert l_2 \Vert_{L^2(0,t;\mathbb{R}^m)} \right).
		\end{aligned}
	\end{equation*}
\end{proposition}
\begin{proof}
	Since the control problem is closely related to the well-studied linear-quadratic control problem, the proof will only be sketched. To see unique solvability of \eqref{LQCP} it needs to be shown that the cost functional is coercive and convex along the kernel of the constraint. This condition is clearly fulfilled for all linear and non-negative quadratic terms. The upper bound on $\Vert \rho \Vert_{L^\infty(0,t;\mathbb{R}^n)}$ ensures that both properties hold for the entire cost functional. The necessary and sufficient optimality condition for $(x,v,q)$ to be the minimizer is then given by \eqref{Appendix KKT system}.
	
	The norm of the control $\Bar{v}$ can then be estimated from above in terms of the cost functional, which in turn is estimated from above by the cost functional evaluated in the zero control and its corresponding trajectory. These estimates can be carried over to the trajectory and adjoint state associated with $\Bar{v}$.
\end{proof} 
\section{Technical auxiliary results}\label{Appendix C}
This section collects auxiliary results and their technical proofs required in this work.
\begin{lemma}\label{Lemma: Bounds for derivatives}
	Let $t \in (0,T]$ and $(\xi, \omega) \in \mathcal{U}_{\delta_4}(0) \subset \mathbb{R}^n \times L^2(0,t;\mathbb{R}^r)$. Let $\mu \in L^2(0,t;\mathbb{R}^r)$ and $z_1,z_2,z_3,z_4 \in \mathbb{R}^n$ and for $i = 1,...,4$ denote $z^i = (z_1,...,z_i)$. Then there exist constants $\overline{M}_1, \overline{M}_2, \overline{M}_3, \overline{M}_4>0$ and $\breve{M}_0, \breve{M}_1,\breve{M}_2 > 0$ independent of $t$, $\xi$, $\omega$, $\mu$, $z_1$, $z_2$, $z_3$ and $z_4$ such that
	\begin{equation*}
		\max \left(
		\Vert D^i_{\xi^i} \mathcal{X}_t(\xi,\omega)z^i \Vert_{V_t},
		\Vert D^i_{\xi^i} \mathcal{P}_t(\xi,\omega)z^i \Vert_{V_t},
		\Vert D^i_{\xi^i} \mathcal{U}_t(\xi,\omega)z^i \Vert_{L^2(0,t;\mathbb{R}^m)} 
		\right)
		\leq \overline{M}_i \prod_{k=1}^i\Vert z_k \Vert
	\end{equation*}
	holds for $i=1,2,3,4$. Further it holds 
	\begin{equation*}
		\max \left(
		\Vert D_{\omega} \mathcal{X}_t(\xi,\omega)\mu \Vert_{V_t},
		\Vert D_{\omega} \mathcal{P}_t(\xi,\omega)\mu \Vert_{V_t},
		\Vert D_{\omega} \mathcal{U}_t(\xi,\omega)\mu \Vert_{L^2(0,t;\mathbb{R}^m)} 
		\right)
		\leq \breve{M}_0 \Vert \mu \Vert_{L^2(0,t;\mathbb{R}^r)},
	\end{equation*}
	and for $i=1,2$ it holds
	\begin{equation*}
		\begin{aligned}
			\max & \left(
			\Vert D^{i+1}_{\omega \xi^i} \mathcal{X}_t(\xi,\omega)(\mu,z^i) \Vert_{V_t},
			\Vert D^{i+1}_{\omega \xi^i} \mathcal{P}_t(\xi,\omega)(\mu,z^i) \Vert_{V_t},
			\Vert D^{i+1}_{\omega \xi^i} \mathcal{U}_t(\xi,\omega)(\mu,z^i) \Vert_{L^2(0,t;\mathbb{R}^m)} 
			\right)\\
			&\leq \breve{M}_i \Vert \mu \Vert_{L^2(0,t;\mathbb{R}^r)} \prod_{k=1}^i\Vert z_k \Vert.
		\end{aligned}
	\end{equation*}
\end{lemma}
\begin{proof}
	First note that it holds $\Vert \mathcal{X}_t(\xi,\omega) \Vert_{L^\infty(0,t;\mathbb{R}^n)} \leq M_1 \delta_4 + \Vert \tilde{x} \Vert_{L^\infty(0,T;\mathbb{R}^n)}$ and \\ $\Vert \mathcal{P}_t(\xi,\omega) \Vert_{L^\infty(0,t;\mathbb{R}^n)} \leq M_2 \delta_4 \leq \beta$. To enforce the second estimate a reduction of $\delta_4$ might be necessary. For the sake of a better overview we abstain from introducing a new variable to reflect this. By construction it holds
	\begin{equation}\label{Phi equation}
		\Phi_t(\mathcal{X}_t(\xi,\omega), \mathcal{U}_t(\xi,\omega), \mathcal{P}_t(\xi,\omega),\omega) = (\omega,0,\xi,0,0,0).
	\end{equation}
	Since $\Phi_t$ is $C^\infty$ on $X_t$ and $\mathcal{X}_t$, $\mathcal{U}_t$ and $\mathcal{P}_t$ are $C^\infty$ on $\mathcal{U}_{\delta_4}(0)$, it follows
	\begin{equation*}
		D_\xi \left( \Phi_t(\mathcal{X}_t(\xi,\omega), \mathcal{U}_t(\xi,\omega), \mathcal{P}_t(\xi,\omega),\omega) \right) z_1 = (0,0,z_1,0,0,0).
	\end{equation*}
	With the chain rule one gets that the triple $(D_\xi \mathcal{X}_t(\xi,\omega)z_1, D_\xi \mathcal{U}_t(\xi,\omega)z_1, D_\xi \mathcal{P}_t(\xi,\omega)z_1)$ solves \eqref{Appendix KKT system} with $\breve{x}_t = \mathcal{X}_t(\xi,\omega)$, $\rho = \mathcal{P}_t(\xi,\omega)$, $f=0$, $a = z_1$, $\mu = 0$, $l_1 = 0$, $b = 0$, $l_2 = 0$. Proposition \ref{Proposition: Appendix A} yields the estimate for the first order $\xi$-derivatives. Taking two $\xi$-derivatives of \eqref{Phi equation} shows that $(D_{\xi^2} \mathcal{X}_t(\xi,\omega)z^2, D_{\xi^2} \mathcal{U}_t(\xi,\omega)z^2, D_{\xi^2} \mathcal{P}_t(\xi,\omega)z^2)$ solves \eqref{Appendix KKT system} with
	\begin{equation*}
		\begin{aligned}
			\breve{x}_t &= \mathcal{X}_t(\xi,\omega),~
			\rho = \mathcal{P}_t(\xi,\omega),~
			a = 0,~ \mu = 0,~
			b = 0,~ l_2 = 0,\\
			f &= G(D_\xi \mathcal{X}_t(\xi,\omega)z_1 \otimes D_\xi \mathcal{X}_t(\xi,\omega)z_2) 
			+ G(D_\xi \mathcal{X}_t(\xi,\omega) z_2 \otimes D_\xi \mathcal{X}_t(\xi,\omega) z_1),\\
			l_1 &= (D_\xi \mathcal{X}_t(\xi,\omega)z_1^\top \otimes I_n)G^\top \mathcal{P}_t(\xi,\omega)z_2 
			+ (I_n \otimes D_\xi \mathcal{X}_t(\xi,\omega)z_1^\top)G^\top \mathcal{P}_t(\xi,\omega)z_2\\ 
			&+ (D_\xi \mathcal{X}_t(\xi,\omega)z_2^\top \otimes I_n)G^\top \mathcal{P}_t(\xi,\omega)z_1 
			+ (I_n \otimes D_\xi \mathcal{X}_t(\xi,\omega)z_2^\top)G^\top \mathcal{P}_t(\xi,\omega)z_1.
		\end{aligned}
	\end{equation*}
	Again the desired estimates are obtained with Proposition \ref{Proposition: Appendix A}.
	The remaining estimates are shown analogously.
\end{proof}
\begin{lemma}\label{Lemma: V differentiable Auxiliary1}
	Let $(\xi,\omega) \in \mathcal{U}_{\tfrac{1}{2}\delta_4}(0) \subset \mathbb{R}^n \times L^2(0,T;\mathbb{R}^r)$. Then for all $t \in (0,T)$ it holds
	\begin{equation*}
		\frac{1}{\tau} \int_t^{t+\tau} \left\Vert \mathcal{X}_{t+\tau}(\xi,\omega)(s) - \mathcal{X}_t(\xi,\omega)(t) \right\Vert \,\mathrm{d}s \to 0,~~~~~\tau \searrow 0.
	\end{equation*}
\end{lemma}
\begin{proof}
	Let $t \in (0,T)$ be fixed and let $\tau > 0$ be small enough such that $t+\tau < T$. Define $\eta(z) = \mathcal{X}_{t+\tau}(\xi,\omega)(t+\tau - z)$. Then 
	\begin{equation}\label{TestEta}
		\dot{\eta}(z) = -A \eta(z) - G(\eta(z) \otimes \eta(z)) - F \mathcal{U}_{t+\tau}(\xi,\omega)(t+\tau-z).
	\end{equation}
	Note that for all $z \in (0,t+\tau-s)$ it holds 
	\begin{equation*}
		\langle \dot{\eta}(z),\eta(z) - \mathcal{X}_t(\xi,\omega)(t) \rangle
		= \frac{\mathrm{d}}{\mathrm{d}z} \frac{1}{2} \Vert \eta(z) - \mathcal{X}_t(\xi,\omega)(t) \Vert^2.
	\end{equation*}
	Hence testing \eqref{TestEta} with $\eta(z) - \mathcal{X}_t(\xi,\omega)(t)$ and integrating over $(0,t+\tau-s)$ for some fixed $s \in [t,t+\tau)$ yields
	\begin{equation*}
		\begin{aligned}
			&\frac{1}{2} \left\Vert \mathcal{X}_{t+\tau}(\xi,\omega)(s) - \mathcal{X}_t(\xi,\omega)(t) \right\Vert^2
			= \frac{1}{2} \left\Vert \eta(t+\tau-s) - \mathcal{X}_t(\xi,\omega)(t) \right\Vert^2\\
			&\leq \frac{1}{2} \left\Vert \eta(0) - \mathcal{X}_t(\xi,\omega)(t) \right\Vert^2 \\
			&+ \int_0^{t+\tau - s} \left( \Vert A \Vert_2 \Vert \eta \Vert + \Vert G \Vert_2 \Vert \eta \Vert^2 + \Vert F \Vert_2^2 \Vert \mathcal{P}_{t+\tau}(\xi,\omega) \Vert \right) \left( \Vert \eta \Vert + \Vert \xi \Vert + \Vert \Tilde{x}(t) \Vert \right) \,\mathrm{d}s\\
			&\leq \frac{1}{2} \Vert \Tilde{x}(t+\tau) - \Tilde{x}(t) \Vert^2 \\
			&+ \tau \left( \Vert A \Vert_2 \Vert \eta \Vert_{L^\infty(0,t+\tau;\mathbb{R}^n)} + \Vert G \Vert_2 \Vert \eta \Vert_{L^\infty(0,t+\tau;\mathbb{R}^n)}^2 \right.  
			+\left. \vphantom{\Vert \eta \Vert_{L^\infty(0,t+\tau;\mathbb{R}^n)}^2} \Vert F \Vert_2^2 \Vert \mathcal{P}_{t+\tau}(\xi,\omega) \Vert_{L^\infty(0,t+\tau;\mathbb{R}^n)}  \right)\\ 
			&\left( \Vert \eta \Vert_{L^\infty(0,t+\tau;\mathbb{R}^n)} + \Vert \xi \Vert + \Vert \Tilde{x}(t) \Vert \right).
		\end{aligned}
	\end{equation*}
	Since all the norms on the right hand side can be bounded uniformly in $\tau$, there exists some constant $c > 0$ such that for all $s \in [t,t+\tau)$ it holds
	\begin{equation*}
		\frac{1}{2} \left\Vert \mathcal{X}_{t+\tau}(\xi,\omega)(s) - \mathcal{X}_t(\xi,\omega)(t) \right\Vert
		\leq \frac{1}{\sqrt{2}} \Vert \Tilde{x}(t+\tau) - \Tilde{x}(t) \Vert + \sqrt{\tau} c.
	\end{equation*}
	Hence
	\begin{equation*}
		\begin{aligned}
			&\frac{1}{\tau} \int_t^{t+\tau} \left\Vert \mathcal{X}_{t+\tau}(\xi,\omega)(s) - \mathcal{X}_t(\xi,\omega)(t) \right\Vert \,\mathrm{d}s\\
			&\leq \frac{1}{\tau} \int_t^{t+\tau} \sqrt{2} \Vert \Tilde{x}(t+\tau) - \Tilde{x}(t) \Vert + 2\sqrt{\tau} c \,\mathrm{d}s
			= \sqrt{2} \Vert \Tilde{x}(t+\tau) - \Tilde{x}(t) \Vert + 2 \sqrt{\tau} c.
		\end{aligned}
	\end{equation*}
	The continuity of $\Tilde{x}$ yields the assertion.
\end{proof}
\begin{lemma}\label{Lemma: V differentiable Auxiliary2}
	Let $(\xi,\omega) \in \mathcal{U}_{\tfrac{1}{2}\delta_4}(0) \subset \mathbb{R}^n \times L^2(0,T;\mathbb{R}^r)$ and assume $\omega$ to be continuous. Then for all $t \in (0,T)$ it holds
	\begin{equation*}
		\frac{1}{\tau} \int_t^{t+\tau} \left\Vert \mathcal{P}_{t+\tau}(\xi,\omega)(s) - \mathcal{P}_t(\xi,\omega)(t) \right\Vert \,\mathrm{d}s \rightarrow 0,~~~~~\tau \searrow 0.
	\end{equation*}
\end{lemma}
\begin{proof}
	The proof is done analogously to the proof of Lemma \ref{Lemma: V differentiable Auxiliary1} using the transformation $\eta(z) = \mathcal{P}_{t+\tau}(\xi,\omega)(t+\tau-z)$. The initial term of this transformation does not lead to a term involving $\tilde{x}$ but instead $  \mathcal{P}_{t + \tau}(\xi,\omega)(t+\tau) - \mathcal{P}_t(\xi,\omega)(t) $. With Lemma \ref{Corollary: Feedback Representation Optimal Control} this is equal to $ - \nabla_\xi \mathcal{V}(t+\tau,\xi,\omega) + \nabla_\xi \mathcal{V}(t,\xi,\omega)$. Proposition \ref{Proposition: Space derivatives continuous in time} ensures the convergence to zero. 
\end{proof}
\begin{lemma}\label{Lemma: V differentiable Auxiliary3}
	Let $(\xi,\omega) \in \mathcal{U}_{\tfrac{1}{2}\delta_4}(0) \subset \mathbb{R}^n \times L^2(0,T;\mathbb{R}^r)$ and assume $\omega$ to be continuous. Then for all $t \in (0,T)$ it holds
	\begin{equation*}
		\frac{1}{\tau} \int_t^{t+\tau} \dot{\mathcal{X}}_{t+\tau}(\xi,\omega)(s) \,\mathrm{d}s \rightarrow \dot{\mathcal{X}}_t(\xi,\omega)(t),~~~~~\tau \searrow 0.
	\end{equation*}
\end{lemma}
\begin{proof}
	Let $t \in (0,T)$ be fixed and let $\tau > 0 $ be small enough such that $t+\tau < T$. Then
	\begin{equation*}
		\begin{aligned}
			&\left\Vert \frac{1}{\tau} \int_t^{t+\tau} \dot{\mathcal{X}}_{t+\tau}(\xi,\omega)(s) \,\mathrm{d}s - \dot{\mathcal{X}}_t(\xi,\omega)(t) \right\Vert\\
			&\leq \frac{1}{\tau} \int_t^{t+\tau}
			\Vert A(\mathcal{X}_{t+\tau}(\xi,\omega)(s) - \mathcal{X}_t(\xi,\omega)(t)) \Vert
			+ \Vert F (\mathcal{U}_{t+\tau}(\xi,\omega)(s) - \mathcal{U}_t(\xi,\omega)(t)) \Vert\\
			&+ \Vert G (\mathcal{X}_{t+\tau}(\xi,\omega)(s) \otimes ( \mathcal{X}_{t+\tau}(\xi,\omega)(s) - \mathcal{X}_t(\xi,\omega)(t)) ) \Vert \\
			&+ \Vert G (( \mathcal{X}_{t+\tau}(\xi,\omega)(s) - \mathcal{X}_t(\xi,\omega)(t) ) \otimes \mathcal{X}_t(\xi,\omega)(t)) \Vert
			\,\mathrm{d}s\\
			&\leq \left( \Vert A \Vert_2 + \Vert G \Vert_2 \Vert \mathcal{X}_{t+\tau}(\xi,\omega) \Vert_{L^\infty(0,t+\tau;\mathbb{R}^n)} + \Vert G \Vert_2 \Vert \mathcal{X}_{t}(\xi,\omega)(t) \Vert \right)\\
			&\frac{1}{\tau} \int_t^{t+\tau} \Vert \mathcal{X}_{t+\tau}(\xi,\omega)(s) - \mathcal{X}_t(\xi,\omega)(t) \Vert \,\mathrm{d}s
			+ \Vert F \Vert_2^2 \frac{1}{\tau} \int_t^{t+\tau} \Vert \mathcal{P}_{t+\tau}(\xi,\omega)(s) - \mathcal{P}_t(\xi,\omega)(t) \Vert \,\mathrm{d}s.
		\end{aligned}
	\end{equation*}
	The assertion follows with Lemma \ref{Lemma: V differentiable Auxiliary1}, Lemma \ref{Lemma: V differentiable Auxiliary2} and the uniform bound on \\ $ \Vert \mathcal{X}_{t+\tau}(\xi,\omega) \Vert_{L^\infty(0,t+\tau;\mathbb{R}^n)} $ derived from Proposition \ref{Proposition: Existence}.
\end{proof}
\begin{lemma}\label{Lemma: V differentiable Auxiliary4}
	Let $(\xi,\omega) \in \mathcal{U}_{\tfrac{1}{2}\delta_4}(0) \subset \mathbb{R}^n \times L^2(0,T;\mathbb{R}^r)$ and assume $\omega$ to be continuous. Then for all $t \in (0,T)$ and $\tau \searrow 0$ it holds
	\begin{equation*}
	\begin{aligned}
		&\frac{1}{\tau} \int_t^{t+\tau} \left\Vert \mathcal{U}_{t+\tau}(\xi,\omega) \right\Vert^2 
		+ \alpha \left\Vert \omega - C(\mathcal{X}_{t+\tau}(\xi,\omega) - \Tilde{x}) \right\Vert^2 \,\mathrm{d}s \\
		&\rightarrow \left\Vert \mathcal{U}_{t}(\xi,\omega)(t) \right\Vert^2 
		+ \alpha \left\Vert \omega(t) - C(\mathcal{X}_{t}(\xi,\omega)(t) - \Tilde{x}(t)) \right\Vert^2.
	\end{aligned}
	\end{equation*}
\end{lemma}
\begin{proof}
	Let $t\in (0,T)$ be fixed and let $\tau >0$ be small enough such that $t+\tau < T$. For the first summand consider
	\begin{equation*}
		\begin{aligned}
			&\left\vert \frac{1}{\tau} \int_t^{t+\tau} \left\Vert \mathcal{U}_{t+\tau}(\xi,\omega)(s) \right\Vert^2 \,\mathrm{d}s - \left\Vert \mathcal{U}_t(\xi,\omega)(t) \right\Vert^2 \right\vert\\
			&\leq \frac{1}{\tau} \int_t^{t+\tau} \left\vert \left\langle \mathcal{U}_{t+\tau}(\xi,\omega)(s) + \mathcal{U}_t(\xi,\omega)(t) , \mathcal{U}_{t+\tau}(\xi,\omega)(s) - \mathcal{U}_t(\xi,\omega)(t) \right\rangle \right\vert \,\mathrm{d}s\\
			&\leq \left( \Vert F \Vert_2 \Vert \mathcal{P}_{t+\tau}(\xi,\omega) \Vert_{L^\infty(0,t+\tau;\mathbb{R}^n)} + \left\Vert \mathcal{U}_t(\xi,\omega)(t) \right\Vert \right) \\
			&\Vert F \Vert_2 \frac{1}{\tau} \int_t^{t+\tau} \left\Vert \mathcal{P}_{t+\tau}(\xi,\omega)(s) - \mathcal{P}_t(\xi,\omega)(t) \right\Vert \,\mathrm{d}s.
		\end{aligned}
	\end{equation*}
	Lemma \ref{Lemma: V differentiable Auxiliary2} and the uniform bound on $\Vert \mathcal{P}_{t+\tau}(\xi,\omega) \Vert_{L^\infty(0,t+\tau;\mathbb{R}^n)}$ derived from Proposition \ref{Proposition: FirstOrderOptimalityCondition} yield the convergence. For the second summand consider
	\begin{equation*}
		\begin{aligned}
			&\left\vert \frac{1}{\tau} \int_t^{t+\tau} \left\Vert \omega(s) - C \mathcal{X}_{t+\tau}(\xi,\omega)(s) + C \Tilde{x}(s) \right\Vert^2 \,\mathrm{d}s - \left\Vert \omega(t) - C \mathcal{X}_t(\xi,\omega)(t) + C \Tilde{x}(t) \right\Vert^2 \right\vert\\
			&= \left\vert \frac{1}{\tau} \int_t^{t+\tau} \Vert \omega(s) \Vert^2 + \Vert C \Tilde{x}(s) \Vert^2 + 2 \left\langle \omega(s) , C \Tilde{x}(s) \right\rangle \,\mathrm{d}s \right. \\
			&- \left. \Vert \omega(t) \Vert^2 - \Vert C \Tilde{x}(t) \Vert^2 - 2 \left\langle \omega(t) , C \Tilde{x}(t) \right\rangle \vphantom{\frac{1}{\tau} \int_t^{t+\tau}} \right\vert\\
			&+ \left\vert \frac{1}{\tau} \int_t^{t+\tau} \left\Vert C \mathcal{X}_{t+\tau}(\xi,\omega)(s) \right\Vert^2 - \left\Vert C \mathcal{X}_t(\xi,\omega)(t) \right\Vert^2 \right.
			- 2 \left\langle \omega(s) , C \mathcal{X}_{t+\tau}(\xi,\omega)(s) \right\rangle\\ 
			&+ \left. 2 \left\langle \omega(t) , C \mathcal{X}_t(\xi,\omega)(t) \right\rangle - 2 \left\langle C \Tilde{x}(s) , C \mathcal{X}_{t+\tau}(\xi,\omega)(s) \right\rangle + 2 \left\langle C \Tilde{x}(t) , C \mathcal{X}_t(\xi,\omega)(t) \right\rangle \,\mathrm{d}s \vphantom{\int_t^{t+\tau}} \right\vert.
		\end{aligned}
	\end{equation*}
	Since $\omega$ and $\Tilde{x}$ are continuous, the first two summands of the right hand side converge to zero for $\tau \searrow 0$. The remaining part is smaller or equal to
	\begin{equation*}
		\begin{aligned}
			&\frac{1}{\tau} \int_t^{t+\tau} \left( \left\vert \left\langle C \mathcal{X}_{t+\tau}(\xi,\omega)(s) + C \mathcal{X}_t(\xi,\omega)(t) , C \mathcal{X}_{t+\tau}(\xi,\omega)(s) - C \mathcal{X}_t(\xi,\omega)(t) \right\rangle \right\vert \vphantom{\int_t^{t+\tau}} \right.\\
			&+ 2 \left\vert \left\langle \omega(s) , C \mathcal{X}_{t+\tau}(\xi,\omega)(s) - C \mathcal{X}_t(\xi,\omega)(t) \right\rangle \right\vert 
			+ 2 \left\vert \left\langle \omega(t) - \omega(s) , C \mathcal{X}_t(\xi,\omega)(t) \right\rangle \right\vert \\
			&+ 2 \left. \vphantom{\int_t^{t+\tau}} \left\vert \left\langle C \mathcal{X}_{t+\tau}(\xi,\omega)(s) - C \mathcal{X}_t(\xi,\omega)(t) , C \Tilde{x}(s) \right\rangle \right\vert
			+ 2 \left\vert \left\langle C \mathcal{X}_t(\xi,\omega)(t) , C \Tilde{x}(t) - C \Tilde{x}(s) \right\rangle \right\vert \right) \,\mathrm{d}s \\
			&\leq \Vert C \Vert_2^2 \left( \Vert \mathcal{X}_{t+\tau}(\xi,\omega) \Vert_{L^\infty(0,t+\tau;\mathbb{R}^n)} + \Vert \mathcal{X}_t(\xi,\omega)(t) \Vert \right) \\
			&\frac{1}{\tau} \int_t^{t+\tau} \left\Vert \mathcal{X}_{t+\tau}(\xi,\omega)(s) - \mathcal{X}_t(\xi,\omega)(t) \right\Vert \,\mathrm{d}s\\
			&+ 2 \Vert \omega \Vert_{L^\infty(0,T;\mathbb{R}^r)} \Vert C \Vert_2 \frac{1}{\tau} \int_t^{t+\tau} \left\Vert \mathcal{X}_{t+\tau}(\xi,\omega)(s) - \mathcal{X}_t(\xi,\omega)(t) \right\Vert \,\mathrm{d}s\\
			&+ 2 \Vert C \Vert_2 \Vert \mathcal{X}_t(\xi,\omega)(t) \Vert \left( \frac{1}{\tau} \int_t^{t+\tau} \Vert \omega(s) - \omega(t) \Vert \,\mathrm{d}s \right) \\
			&2 \Vert C \Vert_2^2 \Vert \mathcal{X}_t(\xi,\omega)(t) \Vert \left( \frac{1}{\tau}  \int_t^{t+\tau} \Vert \tilde{x}(s) - \tilde{x}(t) \Vert \,\mathrm{d}s \right)\\
			&+ 2 \Vert C \Vert_2^2 \Vert \tilde{x} \Vert_{L^\infty(0,T;\mathbb{R}^n)} \frac{1}{\tau} \int_t^{t+\tau} \left\Vert \mathcal{X}_{t+\tau}(\xi,\omega)(s) - \mathcal{X}_t(\xi,\omega)(t) \right\Vert \,\mathrm{d}s.
		\end{aligned}
	\end{equation*}
	The assertion follows with the uniform bound on $\Vert \mathcal{X}_{t+\tau}(\xi,\omega) \Vert_{L^\infty(0,t+\tau;\mathbb{R}^n)}$, Lemma \ref{Lemma: V differentiable Auxiliary1} and the continuity of $\omega$ and $\tilde{x}$.
\end{proof}
\begin{lemma}\label{Lemma: Switching Derivatives}
	Let $\omega \in L^2(0,T;\mathbb{R}^r)$ be continuous and satisfy $\Vert \omega \Vert_{L^2(0,T;\mathbb{R}^r)} < \tfrac{1}{2} \delta_4$. Then it holds:
	\begin{itemize}
		\item[(i)] The mappings $(t,\xi) \mapsto \nabla_\xi \mathcal{V}(t,\xi,\omega)$, $(t,\xi) \mapsto \nabla_{\xi\xi}^2 \mathcal{V}(t,\xi,\omega)$ and $(t,\xi) \mapsto \nabla_{\xi^3}^3 \mathcal{V}(t,\xi,\omega)$ are continuous from $[0,T] \times \mathcal{U}_{\tfrac{1}{2}\delta_4}(0)$ to their respective image spaces $\mathbb{R}^n$, $\mathbb{R}^{n,n}$ and $L(\mathbb{R}^n,\mathbb{R}^{n,n})$.
		\item[(ii)] For any fixed $\xi \in \mathcal{U}_{\tfrac{1}{2}\delta_4}$ and any $i,j \in \{ 1,...,n \}$ the function $s \mapsto \frac{\partial^2}{\partial \xi_j \partial \xi_i} \mathcal{V}(s,\xi,\omega)$ is differentiable in $(0,T)$ and it holds
		\begin{equation*}
			\partial_t \frac{\partial^2}{\partial \xi_j \partial \xi_i} \mathcal{V}(t,\xi,\omega)
			= \frac{\partial^2}{\partial \xi_j \partial \xi_i} \partial_t \mathcal{V}(t,\xi,\omega).
		\end{equation*}
	\end{itemize}
\end{lemma}
\begin{proof}
	(i) We will show continuity for the partial derivatives which then implies (i). To that end let $(t,\xi) \in [0,T] \times \mathcal{U}_{\tfrac{1}{2}\delta_4}(0)$ be arbitrary. Then for all $(\Delta t, \Delta \xi)$ such that $(t+\Delta t, \xi + \Delta \xi) \in [0,T] \times \mathcal{U}_{\tfrac{1}{2}\delta_4}(0)$ and any $i \in \{ 1,...,n \}$ it holds
	\begin{equation*}
		\begin{aligned}
			&\left\vert \frac{\partial}{\partial \xi_i} \mathcal{V}(t+\Delta t,\xi+ \Delta \xi,\omega) - \frac{\partial}{\partial \xi_i} \mathcal{V}(t,\xi,\omega) \right\vert\\
			&\leq \left\vert \frac{\partial}{\partial \xi_i} \mathcal{V}(t+\Delta t,\xi + \Delta \xi,\omega) - \frac{\partial}{\partial \xi_i} \mathcal{V}(t + \Delta t,\xi,\omega) \right\vert
			+ \left\vert \frac{\partial}{\partial \xi_i} \mathcal{V}(t + \Delta t,\xi,\omega) - \frac{\partial}{\partial \xi_i} \mathcal{V}(t,\xi,\omega) \right\vert.    
		\end{aligned}
	\end{equation*}
	Due to the time continuity of the gradient shown in Proposition \ref{Proposition: Space derivatives continuous in time} the second summand converges to zero for $\Delta t \to 0$. Taylor's Theorem implies existence of $\theta \in [0,1]$ such that the first summand is equal to 
	\begin{equation*}
		\left\vert \nabla_\xi \left[ \frac{\partial}{\partial \xi_i} \mathcal{V}\right](t + \Delta t,\xi + \theta \Delta \xi,\omega) ~\Delta \xi \right\vert.
	\end{equation*}
	For sufficiently small $\Delta \xi$ Proposition \ref{Proposition: Bounds for derivatives of the value function} can be used to estimate this expression by $\check{M}_2 \Vert \Delta \xi \Vert$, which converges to zero for $\Delta \xi \to 0$. Utilizing the results from Proposition \ref{Proposition: Bounds for derivatives of the value function} and Proposition  \ref{Proposition: Space derivatives continuous in time}
	the continuity of the partial derivatives of order two and three can be shown analogously.
	
	(ii)Let $\epsilon > 0$ be small enough such that for any $x,y \in ( - \epsilon, \epsilon)$ it holds $\left\Vert x e_i + y e_j + \xi \right\Vert < \tfrac{1}{2} \delta_4. $
	Define the function
	\begin{equation*}
		f \colon (0,T) \times (-\epsilon,\epsilon)^2 \rightarrow \mathbb{R},~~~ (t,x,y) \mapsto \mathcal{V}(t, x e_i + y e_j + \xi, \omega).
	\end{equation*}
	Step 1:\\
	Show that for any fixed $\bar{t} \in (0,T)$ and $y \in (-\epsilon,\epsilon)$ it holds 
	\begin{equation*}
		\frac{\partial^2}{\partial x \partial t} f(\bar{t},0,y) = \frac{\partial^2}{\partial t \partial x} f(\bar{t},0,y).
	\end{equation*}
	Let $t \in (0,T)$ and $x \in (-\epsilon,\epsilon)$ be arbitrary and note that due to Corollary \ref{Corollary: Value function C infty wrt space} and Theorem \ref{Theorem: HJB holds}
	\begin{equation*}
		\frac{\partial}{\partial x}f(t,x,y) = \frac{\partial}{\partial \xi_i} \mathcal{V}(t,xe_i+ye_j+\xi,\omega) \hspace{3mm} \text{  and  } \hspace{3mm}
		\frac{\partial}{\partial t}f(t,x,y) = \partial_t \mathcal{V}(t,xe_i+ye_j+\xi,\omega)
	\end{equation*}
	exist. Taking a partial derivative of \eqref{HJB rigorous} and noting that $\mathcal{V}$ is $C^\infty$ with respect to $\xi$ in $\mathcal{U}_{\tfrac{1}{2} \delta_4}(0)$ ensures the existence of 
	\begin{equation}\label{Proof Switching Derivatives1}
		\begin{aligned}
			\frac{\partial^2}{\partial x \partial t}f(t,x,y) 
			&= \frac{\partial}{\partial \xi_i} \partial_t \mathcal{V}(t,xe_i+ye_j+\xi,\omega) \\
			&= - \left\langle \frac{\partial}{\partial \xi_i} \nabla_\xi \mathcal{V}(t,xe_i+ye_j+\xi,\omega) , h(t,xe_i+ye_j+\xi) \right\rangle\\
			&- \left\langle \nabla_\xi \mathcal{V}(t,xe_i+ye_j+\xi,\omega) , \frac{\partial}{\partial \xi_i} h(t,xe_i + ye_j +\xi) \right\rangle\\
			&- \left\langle F^\top \frac{\partial}{\partial \xi_i} \nabla_\xi \mathcal{V}(t,xe_i+ye_j+\xi,\omega) , F^\top \nabla_\xi \mathcal{V}(t,xe_i+ye_j+\xi,\omega) \right\rangle\\
			&+ \alpha \left\langle C^\top Ce_i , xe_i+ye_j+\xi \right\rangle
			- \alpha \left\langle C^\top \omega(t) , e_i \right\rangle.
		\end{aligned}
	\end{equation}
	It remains to show that this function is continuous in $(t,x) = (\bar{t},0)$, where $\bar{t} \in (0,T)$ is arbitrary. This will be achieved by an examination of the terms on the right hand side. With the assumption $\epsilon$ (i) yields the continuity of the first and second order spatial derivatives. Continuity of the remaining terms in \eqref{Proof Switching Derivatives1} can be seen easily. Hence \cite[9.41 Theorem]{Rud76} yields the existence of $ \frac{\partial^2}{\partial t \partial x}f(\bar{t},0,y) $ and it holds $ \frac{\partial^2}{\partial x \partial t}f(\bar{t},0,y) = \frac{\partial^2}{\partial t \partial x}f(\bar{t},0,y) $. The derivative of the HJB equation ensures that $\frac{\partial^2}{\partial x \partial t} f(\bar{t},0,y)$ is differentiable in $(-\epsilon,\epsilon)$ with respect to $y$. Since the derived equation holds for all $y \in (-\epsilon,\epsilon)$ it follows
	\begin{equation}\label{Proof Switching Derivatives2}
		\frac{\partial^3}{\partial y \partial x \partial t}f(\bar{t},0,y) = \frac{\partial^3}{\partial y \partial t \partial x}f(\bar{t},0,y).
	\end{equation}
	Step 2:\\
	Show that for any $\bar{t} \in (0,T)$ it holds
	\begin{equation}\label{Proof Switching Derivatives3}
		\frac{\partial^3}{\partial y \partial t \partial x}f(\bar{t},0,0) = \frac{\partial^3}{\partial t \partial y \partial x}f(\bar{t},0,0).
	\end{equation}
	This will be achieved by an application of \cite[9.41 Theorem]{Rud76} to the function 
	\begin{equation*}
		(0,T) \times (-\epsilon,\epsilon) \rightarrow \mathbb{R},~~~ (t,y) \mapsto \frac{\partial}{\partial x} f(t,0,y).
	\end{equation*}
	Let $t \in (0,T)$ and $y \in (-\epsilon, \epsilon)$ be arbitrary. Since $\mathcal{V}$ is $C^\infty$ with respect to $\xi$ on an appropriate neighborhood of zero, the partial derivative
	\begin{equation*}
		\frac{\partial}{\partial y} \frac{\partial}{\partial x} f(t,0,y) = \frac{\partial^2}{\partial \xi_j \partial \xi_i} \mathcal{V}(t, y e_j + \xi, \omega)
	\end{equation*}
	exists.
	Furthermore Step 1 ensures the existence of the two partial derivatives
	\begin{equation*}
		\frac{\partial}{\partial t} \frac{\partial}{\partial x}f(t,0,y)~~~~~\text{and}~~~~~
		\frac{\partial^2}{\partial y \partial t} \frac{\partial}{\partial x}f(t,0,y).
	\end{equation*}
	It remains to show that the function 
	\begin{equation*}
		(t,y) \mapsto \frac{\partial^3}{\partial y \partial t \partial x} f(t,0,y)
	\end{equation*}
	is continuous in $(t,y) = (\bar{t},0)$, where $\bar{t} \in (0,T)$ is arbitrary. To that end form derivatives with respect to $\xi_i$ and $\xi_j$ of \eqref{HJB rigorous} to obtain
	\begin{equation}
		\begin{aligned}\label{Proof Switching Derivatives4} 
			&\frac{\partial^3}{\partial y \partial t \partial x} f(t,0,y) 
			= \frac{\partial^3}{\partial y \partial x \partial t} f(t,0,y)
			= \frac{\partial^2}{\partial \xi_j \partial \xi_i} \partial_t \mathcal{V}(t, y e_j + \xi, \omega)\\
			&= - \left\langle \frac{\partial^2}{\partial \xi_j \partial \xi_i} \nabla_\xi \mathcal{V}(t, y e_j + \xi,\omega) , h(t, y e_j + \xi) \right\rangle \\
			&- \left\langle \frac{\partial}{\partial \xi_i} \nabla_\xi \mathcal{V}(t, y e_j + \xi,\omega) , \frac{\partial}{\partial \xi_j} h(t, y e_j + \xi) \right\rangle\\
			&- \left\langle \frac{\partial}{\partial \xi_j} \nabla_\xi \mathcal{V}(t, y e_j + \xi,\omega) , \frac{\partial}{\partial \xi_i} h(t, y e_j + \xi) \right\rangle
			- \left\langle \nabla_\xi \mathcal{V}(t, y e_j + \xi,\omega) , \frac{\partial^2}{\partial \xi_j \partial \xi_i} h(t, y e_j + \xi) \right\rangle\\
			&- \left\langle F^\top \frac{\partial^2}{\partial \xi_j \partial \xi_i} \nabla_\xi \mathcal{V}(t, y e_j + \xi,\omega) , F^\top \nabla_\xi \mathcal{V}(t, y e_j + \xi,\omega) \right\rangle\\
			&- \left\langle F^\top \frac{\partial}{\partial \xi_i} \nabla_\xi \mathcal{V}(t, y e_j + \xi,\omega) , F^\top \frac{\partial}{\partial \xi_j} \nabla_\xi \mathcal{V}(t, y e_j + \xi,\omega) \right\rangle
			+ \alpha \left\langle C^\top C e_i , e_j \right\rangle
		\end{aligned}
	\end{equation}
	and consider the term on the right hand side. Their continuity can be shown analogously to Step 1.
	Therefore $(t,y) \mapsto \frac{\partial^3}{\partial y \partial t \partial x} f(t,0,y)$ is continuous in $(\bar{t},0)$ and \cite[9.41 Theorem]{Rud76} implies (\ref{Proof Switching Derivatives3}).\\
	Step 3:\\
	Now (\ref{Proof Switching Derivatives2}) with $y = 0$ together with (\ref{Proof Switching Derivatives3}) yields
	\begin{equation*}
		\frac{\partial^2}{\partial \xi_j \partial \xi_i} \partial_t \mathcal{V}(\bar{t},\xi,\omega)
		=\frac{\partial^3}{\partial y \partial x \partial t}f(\bar{t},0,0) = \frac{\partial^3}{\partial t \partial y \partial x}f(\bar{t},0,0)
		= \partial_t \frac{\partial^2}{\partial \xi_j \partial \xi_i} \mathcal{V}(\bar{t},\xi,\omega).
	\end{equation*}
\end{proof}
\begin{lemma}\label{lem: EstHesInv}
	Let $t \in [0,T]$ and let $(\xi,\omega) \in \mathcal{U}_{\delta_5}(0) \subset \mathbb{R}^n \times L^2(0,t;\mathbb{R}^r)$. Then there exists a constant $M_\mathcal{V} > 0$ independent of $t$, $\xi$ and $\omega$ such that
	\begin{equation*}
		\Vert \nabla^2_{\xi\xi} \mathcal{V}(t,\xi,\omega)^{-1} \Vert_2 < M_\mathcal{V}.
	\end{equation*}
\end{lemma}
\begin{proof}
	First note that $\delta_5$ was chosen such that for all $(\xi,\omega) \in \mathcal{U}_{\delta_5}(0)$ it holds $\Vert \nabla^2_{\xi\xi} \mathcal{V}(t,\xi,\omega) - \nabla^2_{\xi\xi} \mathcal{V}(t,0,0) \Vert_2 \leq \frac{1}{2} \Vert \nabla^2_{\xi\xi} \mathcal{V}(t,0,0)^{-1} \Vert_2^{-1} $. It follows
	\begin{equation*}
		\Vert \nabla^2_{\xi\xi} \mathcal{V}(t,\xi,\omega)^{-1} \Vert_2 < 2 \Vert \nabla^2_{\xi\xi} \mathcal{V}(t,0,0)^{-1} \Vert_2.
	\end{equation*}
	This can be seen by utilizing the Neumann series, see for example \cite[Proposition 1 in Section 3.5]{Zei95AMS109} and \cite[Proposition 7 in Section1.23]{Zei95AMS108}. As mentioned in the proof of Theorem \ref{Theorem: Hessian invertible close to zero} the mapping $s \mapsto \nabla^2_{\xi\xi} \mathcal{V}(s,0,0)^{-1}$ is continuous in $[0,T]$. Hence the right hand side admits a maximum and setting $M_\mathcal{V} = 2 \max\limits_{s \in [0,T]} \Vert \mathcal{V}(s,0,0)^{-1} \Vert_2$ yields the assertion.
\end{proof}
\begin{lemma}\label{lem: HesInvLip}
	Let $t \in [0,T]$. Then there exists a constant $L_\mathcal{V} > 0$ independent of $t$ such that for all $x,y \in \mathbb{R}^n$ and $\omega \in L^2(0,t;\mathbb{R}^r)$ satisfying $\Vert x \Vert, \Vert y \Vert < \delta_{5}$ and $\Vert \omega \Vert_{L^2(0,t;\mathbb{R}^r)} < \delta_{5}$ it holds
	\begin{equation*}
		\Vert \nabla^2_{\xi\xi} \mathcal{V}(t,y,\omega)^{-1} - \nabla^2_{\xi\xi} \mathcal{V}(t,x,\omega)^{-1} \Vert_2
		\leq L_{\mathcal{V}} \Vert y - x \Vert.
	\end{equation*}
\end{lemma}
\begin{proof}
	Note that $\mathcal{V}(t,\cdot,\omega)$, det$(\cdot)$ and adj$(\cdot)$ are $C^{\infty}$-functions on their respective domains, where adj($A$) is the adjugate matrix of $A$. Hence Cramer's rule implies that $z \mapsto \nabla^2_{\xi\xi} \mathcal{V}(t,z,\omega)^{-1}$ is $C^{\infty}$ in $\mathcal{U}_{\delta_5}(0)$. Then for $x,y \in \mathbb{R}^n$ satisfying $\Vert x \Vert, \Vert y \Vert < \delta_5$ it holds
	\begin{equation*}
		\begin{aligned}
			&\Vert \nabla^2_{\xi\xi}\mathcal{V}(t,y,\omega)^{-1} - \nabla^2_{\xi\xi}\mathcal{V}(t,x,\omega)^{-1} \Vert_2
			= \left\Vert \int_0^1 D_\xi \nabla^2_{\xi\xi}\mathcal{V}(t,x + \tau (y-x),\omega)^{-1} (y-x) \,\mathrm{d}\tau \right\Vert_2\\
			&\leq \int_0^1 \Vert D_\xi \nabla^2_{\xi\xi}\mathcal{V}(t,x + \tau (y-x),\omega)^{-1} \Vert_{L(\mathbb{R}^n,\mathbb{R}^{n,n})} \,\mathrm{d}\tau~ \Vert y-x \Vert\\
			&= \int_0^1 \sup_{\Vert z \Vert = 1} \Vert D_\xi \nabla^2_{\xi\xi}\mathcal{V}(t,x + \tau (y-x),\omega)^{-1} z \Vert_2 \,\mathrm{d}\tau~ \Vert y-x \Vert\\
			&\leq \int_0^1 \sup_{\Vert z \Vert = 1} \sum_{i=1}^n \vert z_i \vert \Vert D_\xi \nabla^2_{\xi\xi}\mathcal{V}(t,x + \tau (y-x),\omega)^{-1} e_i \Vert_2 \,\mathrm{d}\tau \, \Vert y-x \Vert\\
			&\leq \int_0^1 \sum_{i=1}^n \left\Vert \frac{\partial}{\partial \xi_i} \nabla^2_{\xi\xi} \mathcal{V}(t, x + \tau (y-x),\omega)^{-1} \right\Vert_2 \,\mathrm{d}\tau \, \Vert y-x \Vert\\
			&\leq \int_0^1 \sum_{i=1}^n \left\Vert \nabla^2_{\xi\xi} \mathcal{V}(t, x + \tau (y-x),\omega)^{-1} \right\Vert_2^2 \left\Vert \frac{\partial}{\partial \xi_i} \nabla^2_{\xi\xi} \mathcal{V}(t, x + \tau (y-x),\omega) \right\Vert_2 \,\mathrm{d}\tau \, \Vert y-x \Vert\\
			&\leq \int_0^1 \sum_{i=1}^n M_\mathcal{V}^2 \left\Vert \frac{\partial}{\partial \xi_i} \nabla^2_{\xi\xi} \mathcal{V}(t, x + \tau (y-x),\omega) \right\Vert_2 \,\mathrm{d}\tau \, \Vert y-x \Vert,
		\end{aligned}
	\end{equation*}
	where Lemma \ref{lem: EstHesInv} was used in the last estimate. It remains to show an appropriate estimate for\\ $ \left\Vert \frac{\partial}{\partial \xi_i} \nabla^2_{\xi\xi} \mathcal{V}(t, y + \tau (x-y),\omega) \right\Vert_2 $. To that end let $\nu \in \mathbb{R}^n$ fulfill $\Vert \nu \Vert < \delta_{5}$ and let $j,k \in \{  1,...,n \}$ be such that
	\begin{equation*}
		\max\limits_{h,l \in \{  1,...,n \}} \left\vert \frac{\partial^3}{\partial \xi_i \partial \xi_h \partial \xi_l} \mathcal{V}(t,\nu,\omega) \right\vert
		= \left\vert \frac{\partial^3}{\partial \xi_i \partial \xi_j \partial \xi_k} \mathcal{V}(t,\nu,\omega) \right\vert.
	\end{equation*}
	With Proposition \ref{Proposition: Bounds for derivatives of the value function} it follows
	\begin{equation*}
		\begin{aligned}
			&\left\Vert \frac{\partial}{\partial \xi_i} \nabla^2_{\xi\xi} \mathcal{V}(t,\nu,\omega) \right\Vert_2 
			\leq n \left\Vert \frac{\partial}{\partial \xi_i} \nabla^2_{\xi\xi} \mathcal{V}(t,\nu,\omega) \right\Vert_{n,n} 
			= n \left\vert \frac{\partial^3}{\partial \xi_i \partial \xi_j \partial \xi_k} \mathcal{V}(t,\nu,\omega) \right\vert\\
			&= n \left\vert D_\xi^3 \mathcal{V}(t,\nu,\omega)(e_i,e_j,e_k) \right\vert
			\leq n \check{M}_3.
		\end{aligned}
	\end{equation*}
	Finally it follows
	\begin{equation*}
		\Vert \nabla^2_{\xi\xi}\mathcal{V}(t,y,\omega)^{-1} - \nabla^2_{\xi\xi}\mathcal{V}(t,x,\omega)^{-1} \Vert_2
		\leq M_\mathcal{V}^2 n^2 \check{M}_3 \Vert x-y\Vert
		\eqqcolon L_\mathcal{V} \Vert x-y \Vert.
	\end{equation*}
\end{proof}

%%-----------------------------
\bibliographystyle{siam}
\bibliography{references} 
%%-----------------------------

\end{document}